\documentclass[11pt,leqno]{amsart}

\usepackage[a4paper, right=20mm, left=20mm, top=20mm]{geometry}

\usepackage{emptypage}
\usepackage{amsthm}
\usepackage[english]{babel}
\usepackage[utf8]{inputenc}
\usepackage{hyperref}

\usepackage{amsmath}
\usepackage{amssymb}
\usepackage{mathtools}
\usepackage{mathrsfs}
\usepackage{enumerate}

\usepackage{tikz}
\usetikzlibrary{matrix,arrows,positioning}

\numberwithin{equation}{section}
\numberwithin{figure}{section}

\theoremstyle{plain}
\newtheorem{Theorem}{Theorem}[section]
\newtheorem{Lemma}[Theorem]{Lemma}
\newtheorem{Proposition}[Theorem]{Proposition}
\newtheorem{Corollary}[Theorem]{Corollary}
\newtheorem{TheoremAlph}{Theorem}

\theoremstyle{definition}
\newtheorem{Definition}[Theorem]{Definition}
\newtheorem*{Definition*}{Definition}
\newtheorem{Remark}[Theorem]{Remark}

\newcommand{\EatDot}[1]{}

\newcommand{\C}{\mathbb{C}}
\newcommand{\R}{\mathbb{R}}

\newcommand{\Z}{\mathbb{Z}}

\newcommand{\kk}{\Bbbk}

\newcommand{\G}{\mathbf{G}}
\newcommand{\B}{\mathbf{B}}

\newcommand{\T}{\mathbf{T}}

\newcommand{\Hom}{\mathrm{Hom}}

\newcommand{\Ext}{\mathrm{Ext}}
\newcommand{\id}{\mathrm{id}}

\newcommand{\Cdot}{\boldsymbol{\cdot}}
\newcommand{\Stab}{\mathrm{Stab}}
\newcommand{\Rep}{\mathrm{Rep}}
\newcommand{\Tilt}{\mathrm{Tilt}}
\newcommand{\GL}{\mathrm{GL}}
\newcommand{\gfd}{\mathrm{gfd}}

\newcommand{\pr}{\mathrm{pr}}

\newcommand{\rad}{\mathrm{rad}}
\newcommand{\regRep}{\mathrm{\underline{Re}p}}
\newcommand{\Ver}{\mathrm{Ver}}

\providecommand{\abs}[1]{\lvert#1\rvert}

\title[Linkage and translation for tensor products]{Linkage and translation\\for tensor products of representations\\of simple algebraic groups and quantum groups}
\author{Jonathan Gruber}
\address{Department of Mathematics, FAU Erlangen-Nuremberg, 91058 Erlangen, Germany}
\email{jonathan.gruber@fau.de}
\date{\today}

\begin{document}

\begin{abstract}
	Let $\G$ be either a simple linear algebraic group over an algebraically closed field of characteristic $\ell>0$ or a quantum group at an $\ell$-th root of unity.
	We define a tensor ideal of \emph{singular $\G$-modules} in the category $\Rep(\G)$ of finite-dimensional $\G$-modules and study the associated quotient category $\regRep(\G)$, called the \emph{regular quotient}.
	For $\ell \geq h$, the Coxeter number of $\G$, we establish a `linkage principle' and a `translation principle' for tensor products:
	Let $\regRep_0(\G)$ be the essential image in $\regRep(\G)$ of the principal block of $\Rep(\G)$.
	We first show that $\regRep_0(\G)$ is closed under tensor products in $\regRep(\G)$.
	Then we prove that the monoidal structure of $\regRep(\G)$ is governed to a large extent by the monoidal structure of $\regRep_0(\G)$.
	These results can be combined to give an external tensor product decomposition $\regRep(\G) \cong \Ver(\G) \boxtimes \regRep_0(\G)$, where $\Ver(\G)$ denotes the Verlinde category of $\G$.
\end{abstract}

\maketitle

\section*{Introduction}

In the representation theory of groups and of Hopf algebras,
it is often helpful to decompose categories of representations as a direct sum of blocks.
One can then hope to obtain stronger or more fine-grained results by considering one block of the category at a time.
However, this strategy is generally not well suited for understanding the monoidal structure of the category, because a tensor product of two representations, each belonging to a given block, may have indecomposable direct summands in many different blocks.
The main results of this article provide a way of partially overcoming this obstacle, for categories of representations of simple algebraic groups (over fields of positive characteristic) and quantum groups (at roots of unity).
More precisely, we use minimal tilting complexes (see \cite{GruberMinimalTilting}) to define a tensor ideal of \emph{singular modules} in the representation categories.
When considering representations modulo this tensor ideal, it turns out that the principal block is closed under tensor products and that the monoidal structure of the entire category is governed to a large extent by the resulting monoidal structure on the principal block.
We refer to these results as a \emph{linkage principle} and a \emph{translation principle} for tensor products, in analogy with the classical results describing the block decomposition of the categories in question (due to H.H.\ Andersen \cite{AndersenLinkage} and J.C.\ Jantzen \cite{JantzenZurCharakterformel}).
In the following, we briefly recall linkage and translation before discussing our results in more detail.
\medskip

The categories of (finite-dimensional) representations of simple algebraic groups and of quantum groups have many structural properties in common, which  often makes it possible to treat the two cases simultaneously.
We refer to the representation theory of simple algebraic groups as the \emph{modular case} and to the representation theory of quantum groups as the \emph{quantum case}, and we
fix the following notational conventions:\smallbreak

\noindent
\begin{tabular}{p{0.22\textwidth}p{0.7\textwidth}}
	\textbf{The modular case} & Here $\G$ is a simply connected simple linear algebraic group over an algebraically closed field of characteristic $\ell>0$.
	We write $\Rep(\G)$ for the category of finite-dimensional rational $\G$-modules. \vspace{.1cm}
\end{tabular}

\begin{tabular}{p{0.22\textwidth}p{0.7\textwidth}}	
	\textbf{The quantum case} & Here $\G = U_\zeta(\mathfrak{g})$ is the specialization at a complex $\ell$-th root of unity $\zeta$ of Lusztig's divided powers version of the quantum group corresponding to a complex simple Lie algebra $\mathfrak{g}$.
	We write $\Rep(\G)$ for the category of finite-dimensional $\G$-modules of type $1$.
\end{tabular}
\smallbreak
\noindent
In either of the two cases, $\G$ comes equipped with a simple root system $\Phi$ and a weight lattice $X$.
For this introduction (and for most of this article), we suppose that $\ell \geq h$, the Coxeter number of $\Phi$.
In the quantum case, we further assume that $\ell$ is odd (and not divisible by $3$ if $\Phi$ is of type $\mathrm{G}_2$).
From now on, we use the term \emph{$\G$-module} to refer to the objects of $\Rep(\G)$; in particular, all $\G$-modules that we consider are implicitly assumed to be finite-dimensional.
The category $\Rep(\G)$ is a highest weight category with weight poset $X^+ \subseteq X$ the set of dominant weights with respect to a fixed positive system $\Phi^+ \subseteq \Phi$, and we denote by
\[ L(\lambda) , \qquad \Delta(\lambda) , \qquad \nabla(\lambda) , \qquad T(\lambda) \]
the simple $\G$-module, the standard module, the costandard module and the indecomposable tilting $\G$-module of highest weight $\lambda \in X^+$.

Let $W_\mathrm{fin}$ and $W_\mathrm{aff} = \Z\Phi \rtimes W_\mathrm{fin}$ be the finite Weyl group and the affine Weyl group of $\Phi$, respectively, and denote the natural embedding $\Z\Phi \to W_\mathrm{aff}$ by $\gamma \mapsto t_\gamma$.
We consider the \emph{$\ell$-dilated dot action} of $W_\mathrm{aff}$ on $X$, defined by
\[ t_\gamma w \Cdot \lambda = w( \lambda + \rho ) - \rho + \ell\gamma \]
for $\gamma \in \Z\Phi$, $w \in W_\mathrm{fin}$ and $\lambda \in X$, where $\rho = \frac12 \sum_{\alpha\in\Phi^+} \alpha$.
Finally, we write $C_\mathrm{fund}$ for the fundamental alcove in $X_\R = X \otimes_\Z \R$ (see Subsection \ref{subsec:linkagetranslation}).
For a weight $\lambda \in \overline{C}_\mathrm{fund} \cap X$, the \emph{linkage class} $\Rep_\lambda(\G)$ of $\lambda$ is the
full subcategory of $\Rep(\G)$ whose objects are the $\G$-modules all of whose composition factors have highest weight in $X^+ \cap W_\mathrm{aff} \Cdot \lambda$.
The \emph{linkage principle} asserts that $\Rep(\G)$ admits a decomposition
\[ \Rep(\G) = \bigoplus_{\lambda \in \overline{C}_\mathrm{fund} \cap X} \Rep_\lambda(\G) . \]
The linkage class $\Rep_0(\G)$ is called the \emph{principal block} of $\G$.%
\footnote{Not all of the linkage classes are blocks (in the usual sense that they can not be decomposed any further), but those corresponding to weights in $C_\mathrm{fund} \cap X$ are.
A precise description of the blocks of $\Rep(\G)$ can be found in Section II.7.2 of \cite{Jantzen}.}
Furthermore, for $\lambda,\mu \in C_\mathrm{fund} \cap X$, the linkage classes of $\lambda$ and $\mu$ are related via a \emph{translation functor}
\[ T_\lambda^\mu \colon \Rep_\lambda(\G) \longrightarrow \Rep_\mu(\G) , \]
which is an equivalence of categories with quasi-inverse $T_\mu^\lambda$.

Our linkage principle and translation principle for tensor products build on the notion of \emph{singular $\G$-modules}, which we define below, using minimal tilting complexes and negligible tilting modules.
As explained in \cite{GruberMinimalTilting}, associated to every $\G$-module $M$, there is
a unique (up to isomorphism) bounded minimal complex $C_\mathrm{min}(M)$ of tilting $\G$-modules, called the minimal tilting complex of $M$, such that 
\[ H^i\big( C_\mathrm{min}(M) \big) \cong \begin{cases} M & \text{if } i=0 , \\ 0 & \text{if } i \neq 0 . \end{cases} \]
A tilting $\G$-module is called negligible if it has no direct summands $T(\lambda)$ with $\lambda \in C_\mathrm{fund} \cap X$.
It is well-known that the negligible tilting modules form a tensor ideal in the category of tilting $\G$-modules
(see \cite{GeorgievMathieuFusion,AndersenParadowskiFusionCategories}).
Now we are ready to give the key definition:

\begin{Definition*}
	A $\G$-module $M$ is called \emph{singular} if all terms of $C_\mathrm{min}(M)$ are negligible.
	Otherwise, we say that $M$ is \emph{regular}.%
	\footnote{Our terminology is justified by the fact that, for $\lambda \in X^+$, the simple $\G$-module $L(\lambda)$ is regular if and only if its highest weight $\lambda$ is \emph{$\ell$-regular}, i.e.\ if $\lambda \in W_\mathrm{aff} \Cdot \lambda^\prime$ for some $\lambda^\prime \in C_\mathrm{fund} \cap X$ (see Lemma \ref{lem:regularquotientregularweight} below).}
\end{Definition*}

By results of \cite{GruberTensorIdeals}, the singular $\G$-modules form a tensor ideal in $\Rep(\G)$.
For every $\G$-module $M$, we can write
\[ M \cong M_\mathrm{sing} \oplus M_\mathrm{reg} , \]
where $M_\mathrm{sing}$ is the direct sum of all singular indecomposable direct summands of $M$ and where $M_\mathrm{reg}$ is the direct sum of all regular indecomposable direct summands of $M$ (for a fixed Krull--Schmidt decomposition).
By the Krull--Schmidt theorem, the $\G$-modules $M_\mathrm{sing}$ and $M_\mathrm{reg}$ are well-defined up to isomorphism, but there may be no canonical choice of direct sum decomposition $M \cong M_\mathrm{sing} \oplus M_\mathrm{reg}$.
We call $M_\mathrm{sing}$ and $M_\mathrm{reg}$ the \emph{singular part} and the \emph{regular part} of $M$, respectively.

We can now state our linkage principle for tensor products, which asserts that the monoidal structure of $\Rep(\G)$ is compatible with the decomposition into linkage classes, when we consider $\Rep(\G)$ modulo the tensor ideal of singular $\G$-modules.
Suppose for the rest of the introduction that $\ell \geq h$, the Coxeter number of $\G$ (so that $C_\mathrm{fund} \cap X$ is non-empty).

\begin{TheoremAlph} \label{thm:introlinkage}
	Let $\lambda \in C_\mathrm{fund} \cap X$ and let $M$ and $N$ be $\G$-modules in the linkage classes of $0$ and $\lambda$, respectively.
	Then $(M \otimes N)_\mathrm{reg}$ belongs to the linkage class of $\lambda$.
\end{TheoremAlph}

Note that for $\G$-modules $M$ and $N$ in $\Rep_0(\G)$, Theorem \ref{thm:introlinkage} implies that $(M \otimes N)_\mathrm{reg}$ also belongs to $\Rep_0(\G)$.
The next result is our translation principle for tensor products.

\begin{TheoremAlph} \label{thm:introtranslation}
	Let $M$ and $N$ be $\G$-modules in $\Rep_0(\G)$.
	For $\lambda,\mu \in C_\mathrm{fund} \cap X$, we have
	\[ \big( T_0^\lambda M \otimes T_0^\mu N \big)_\mathrm{reg} \cong \bigoplus_{\nu \in C_\mathrm{fund} \cap X} T_0^\nu ( M \otimes N )_{\mathrm{reg}}^{\oplus c_{\lambda,\mu}^\nu} , \]
	where $c_{\lambda,\mu}^\nu = [ T(\lambda) \otimes T(\mu) : T(\nu) ]_\oplus$ for $\nu \in C_\mathrm{fund} \cap X$.
\end{TheoremAlph}

The translation functors $T_0^\lambda$ and $T_0^\mu$ are equivalences for $\lambda,\mu   \in C_\mathrm{fund} \cap X$, so Theorem \ref{thm:introtranslation} implies that the monoidal structure of $\Rep(\G)$ modulo singular $\G$-modules is completely determined by the monoidal structure of $\Rep_0(\G)$ modulo singular $\G$-modules.
Therefore, much like the classical linkage principle and translation principle, the linkage principle and translation principle for tensor products can be used to reduce questions about the structure of tensor products of arbitrary $\G$-modules (modulo singular direct summands) to questions about the structure of tensor products of $\G$-modules in the principal block (see also Remark \ref{rem:translationtensorproduct} below).
We point out that the coefficients $c_{\lambda,\mu}^\nu$ in Theorem \ref{thm:introtranslation} are the structure constants of the Verlinde category $\Ver(\G)$ (i.e.\ the quotient of the category of tilting $\G$-modules by the tensor ideal of negligible tilting modules)
and that they can be computed as an alternating sum of dimensions of weight spaces of Weyl modules (see Subsection \ref{sec:negligible}).

Now let us write $\regRep(\G)$ for the quotient category of $\Rep(\G)$ by the tensor ideal of singular $\G$-modules.
Let $\regRep_0(\G)$ be the essential image of the principal block under the canonical quotient functor $\Rep(\G) \to \regRep(\G)$;
so $\regRep_0(\G)$ is the full subcategory of $\regRep(\G)$ whose objects are the $\G$-modules $M$ such that $M_\mathrm{reg}$ belongs to $\Rep_0(\G)$.
By Theorem \ref{thm:introlinkage}, $\regRep_0(\G)$ is closed under tensor products in $\regRep(\G)$. 
The Theorems \ref{thm:introlinkage} and \ref{thm:introtranslation} can be combined to give the following external tensor product decomposition of $\regRep(\G)$:

\begin{TheoremAlph} \label{thm:introproduct}
	There is an equivalence of $\kk$-linear monoidal categories
	\[ \regRep(\G) \cong \Ver(\G) \boxtimes \regRep_0(\G) \]
\end{TheoremAlph}

Here, the box product denotes the additive closure of the category whose objects are pairs consisting of
a tilting $\G$-module and a $\G$-module with regular part in $\Rep_0(\G)$,
and whose $\Hom$-spaces are tensor products of $\Hom$-spaces in $\Ver(\G)$ and $\regRep_0(\G)$.
See Section \ref{sec:tensorproductdecomposition} for more details.

To conclude the introduction, we indicate how the paper is organized.
In Section \ref{sec:preliminaries}, we set up our notation and summarize relevant results about representations of algebraic groups and quantum groups.
In Section \ref{sec:minimaltilting}, we recall results about minimal tilting complexes from \cite{GruberMinimalTilting,GruberTensorIdeals} and we investigate in detail the minimal tilting complexes of Weyl modules and simple $\G$-modules (see Propositions \ref{prop:minimalcomplexWeylmodule} and \ref{prop:minimalcomplexsimplemodule}).
In Section \ref{sec:singularmodules}, we define the tensor ideal of singular $\G$-modules and study the associated quotient category.
We first prove more functorial versions of Theorems \ref{thm:introlinkage} and \ref{thm:introtranslation} in Lemma \ref{lem:linkageclassesregularquotient} and Theorem \ref{thm:translationtensorquotient} and then prove the versions of these theorems that are stated above in Lemma \ref{lem:regularpartlinkageclass} and Theorem \ref{thm:regularparttensorproduct}.
In Section \ref{sec:strongregularity}, we discuss a notion of strong regularity, which is more well-behaved with respect to tensor products than the notion of regularity defined above.
This allows us to generalize results of A.\ Parker \cite{ParkerGFD} about the good filtration dimension of Weyl modules and simple $\G$-modules to iterated tensor products these modules in Theorem \ref{thm:tensorproductweylmodulesimplemoduleGFD}.
As an application of our results that may be of independent interest, we use singular $\G$-modules in Section \ref{sec:translation} to prove that a composition of translation functors between linkage classes corresponding to weights in $C_\mathrm{fund} \cap X$ is naturally isomorphic to a translation functor.
Finally, in Section \ref{sec:tensorproductdecomposition}, we prove Theorem \ref{thm:introproduct} and discuss some related external tensor product decompositions, and in Section \ref{sec:examples} we give some examples of regular parts of tensor products for $\G$ of type $\mathrm{A}_2$.

\subsection*{Acknowledgements}

This article contains some of the main results of my PhD thesis, and I would like to express my deepest gratitude to my PhD advisor Donna Testerman.
I would also like to thank Stephen Donkin, Johannes Flake, Thorge Jensen and Daniel Nakano for their
many helpful comments.
This work was funded by the Swiss National Science Foundation under the grants FNS 200020\_175571 and FNS 200020\_207730 and by the Singapore MOE grant R-146-000-294-133.

\section{Preliminaries} \label{sec:preliminaries}

\subsection{Roots and weights}

Let $\Phi$
be a simple root system in a euclidean space $X_\R$ with scalar product $( - \,, - )$.
For $\alpha\in\Phi$, we denote by $\alpha^\vee = \frac{2 \alpha}{( \alpha , \alpha )}$ the coroot of $\alpha$. The weight lattice of $\Phi$ is
\[ X = \{ \lambda \in X_\R \mid ( \lambda , \alpha^\vee ) \in \Z \text{ for all } \alpha \in \Phi \} , \]
and the Weyl group of $\Phi$ is the (finite) subgroup $W_\mathrm{fin} = \langle s_\alpha \mid \alpha\in\Phi \rangle$
of $\GL(X_\R)$ generated by the reflections $s_\alpha$, where $s_\alpha(x) = x - (x , \alpha^\vee) \cdot \alpha$ for $x\in X_\R$.
The index of the root lattice $\Z\Phi$ in the weight lattice $X$ is finite, and the quotient $X / \Z\Phi$ is called the fundamental group of $\Phi$.
Now fix a positive system $\Phi^+ \subseteq \Phi$
corresponding to a base $\Pi$
of $\Phi$, and let
\[ X^+ \coloneqq \{ \lambda \in X \mid (\lambda , \alpha^\vee) \geq 0 \text{ for all } \alpha \in \Phi^+ \} \]
be the set of dominant weights
with respect to $\Phi^+$.
We consider the partial order on $X$ that is defined by $\lambda \geq  \mu$ if and only if $\lambda-\mu$ is a non-negative integer linear combination of positive roots.
Furthermore, we write $\tilde{\alpha}_\mathrm{h}$
and $\alpha_\mathrm{h}$
for the highest root and the highest short root in $\Phi^+$, respectively, with the convention that $\tilde{\alpha}_\mathrm{h}=\alpha_\mathrm{h}$ (and that all roots are short) if $\Phi$ is simply laced.
We let $\rho = \frac12 \sum_{\alpha\in\Phi^+} \alpha$
be the half-sum of all positive roots and write $h = ( \rho , \alpha_\mathrm{h}^\vee ) + 1$ for the Coxeter number
of $\Phi$.
The Weyl group $W_\mathrm{fin}$ is a Coxeter group with simple reflections $S_\mathrm{fin} = \{ s_\alpha \mid \alpha\in\Pi \}$, and we write $w_0 \in W_\mathrm{fin}$ for its longest element.

\subsection{Algebraic groups and quantum groups} \label{sec:algebraicgroupsquantumgroups}

The root system $\Phi$ is at the heart of the structure of two kinds of Lie theoretic objects whose finite-dimensional simple modules are canonically indexed by $X^+$: \emph{simple algebraic groups} (over a field of positive characteristic) and \emph{quantum groups} (at a root of unity).
The representation theory of quantum groups parallels that of algebraic groups to a large extent, so we will mostly treat the two cases simultaneously. When a distinction becomes necessary, we refer to the representation theory of the algebraic group as the \emph{modular case} and to the representation theory of the quantum group as the \emph{quantum case}.

\subsubsection*{The modular case}

We follow the notational conventions from Section II.1 in \cite{Jantzen}.
Let $\G_\Z$ be a split simply-connected simple algebraic group scheme over $\Z$ with split maximal torus $\T_\Z$, such that the root system of $\G_\Z$ with respect to $\T_\Z$ is isomorphic to $\Phi$.
The positive system $\Phi^+ \subseteq \Phi$ determines
a Borel subgroup $\B^+_\Z$, and the negative roots $-\Phi^+$ determine a Borel subgroup $\B_\Z$.
We fix an algebraically closed field $\kk$ of characteristic $\ell>0$
and denote by $\G=\G_\kk$
the simply-connected simple algebraic group scheme over $\kk$ corresponding to $\G_\Z$, with maximal torus $\T=\T_\kk$
and Borel subgroup $\B=\B_\kk$.

\subsubsection*{The quantum case}

We follow the conventions of Sections II.H.1--II.H.6 in \cite{Jantzen}.
Let $\mathfrak{g}$ be the complex simple Lie algebra with root system $\Phi$, let $U_q(\mathfrak{g})$ be its quantized enveloping algebra and let $U_q^\Z(\mathfrak{g})$ be Lusztig's integral form of $U_q(\mathfrak{g})$, a $\Z[q,q^{-1}]$-subalgebra of $U_q(\mathfrak{g})$ generated by divided powers.
Now let $\zeta \in \C$ be a primitive $\ell$-th root of unity, where $\ell \in \Z_{\geq 0}$ is odd and $\ell \neq 3$ if $\Phi$ is of type $\mathrm{G}_2$.
Then there is a unique ring homomorphism $\Z[q,q^{-1}] \to \C$ with $q \mapsto \zeta$, and we define
\[ U_\zeta(\mathfrak{g}) = \C \otimes_{\Z[q,q^{-1}]} U_q^\Z(\mathfrak{g}) \]
as the extension of scalars of $U_q^\Z(\mathfrak{g})$ via this homomorphism.
In order to be able to treat the quantum case and the modular case simultaneously, we write $\kk = \C$ and $\G = U_\zeta(\mathfrak{g})$ in the quantum case.

\subsection{Representation categories}

In the modular case, we write $\Rep(\G)$ for the category of finite-dimensional $\G$-modules (in the sense of Section I.2.3 in \cite{Jantzen}), and in the quantum case we write $\Rep(\G)$ for the category of finite-dimensional $\G$-modules of type $1$ (see Section II.H.10 in \cite{Jantzen}).
In either case, we will from now on refer to the objects of $\Rep(\G)$ as \emph{$\G$-modules}.
For two $\G$-modules $M$ and $N$, we write $\Hom_\G(M,N)$ for the space of homomorphisms from $M$ to $N$ in $\Rep(\G)$ and $\Ext_\G^i(M,N)$ for the $\Ext$-groups.
If $N$ is indecomposable then we denote by $[M:N]_\oplus$ the multiplicity of $N$ in a Krull--Schmidt decomposition of $M$.

The category $\Rep(\G)$ is a highest weight category with weight poset $(X^+,\leq)$ and we write
\[ L(\lambda) , \qquad \nabla(\lambda) , \qquad \Delta(\lambda) , \qquad T(\lambda) \]
for the simple $\G$-module, the costandard module, the standard module and the indecomposable tilting $\G$-module of highest weight $\lambda \in X^+$, respectively.
The costandard modules are also called \emph{induced modules} and the standard modules are called \emph{Weyl modules}.
A $\G$-module is said to have a \emph{good filtration} (or a \emph{Weyl filtration}) if it has a filtration with successive quotients isomorphic to induced modules (or Weyl modules), and a tilting $\G$-module is (by definition) a $\G$-module that admits both a good filtration and a Weyl filtration.
Every tilting $\G$-module decomposes as a direct sum of indecomposable tilting $\G$-modules $T(\lambda)$ with $\lambda \in X^+$.
We write $\Tilt(\G)$ for the full subcategory of tilting $\G$-modules in $\Rep(\G)$.

The category $\Rep(\G)$ also has a rigid monoidal structure and a braiding.
In the modular case, this follows from Section I.2.7 in \cite{Jantzen}, and the braiding is induced by the canonical symmetric braiding on the category of $\kk$-vector spaces.
In the quantum case, $\Rep(\G)$ is rigid monoidal because $\G$ is a Hopf algebra, and the braiding is constructed in \cite[Chapter 32]{LusztigQuantumGroups}.
For all $\lambda \in X^+$, we have $\Delta(\lambda)^* \cong \nabla\big(-w_0(\lambda)\big)$, and a tensor product of two induced modules has a good filtration by results of S.\ Donkin \cite{DonkinGoodFiltration}, O.\ Mathieu \cite{MathieuGoodFiltration} and J.\ Paradowski \cite{ParadowskiGoodFiltration}.
This also implies that $\Tilt(\G)$ is closed under tensor products in $\Rep(\G)$.

\subsection{Affine Weyl group and alcove geometry}

The affine Weyl group $W_\mathrm{aff}$ and the extended affine Weyl group $W_\mathrm{ext}$ of $\G$ are defined by
\[ W_\mathrm{aff} = \Z\Phi \rtimes W_\mathrm{fin} \qquad \text{and} \qquad W_\mathrm{ext} = X \rtimes W_\mathrm{fin} . \]
The canonical embedding of $X$ into $W_\mathrm{ext}$ is denoted by $\lambda \mapsto t_\lambda$.
Since the action of $W_\mathrm{fin}$ on $X / \Z\Phi$ is trivial, $W_\mathrm{aff}$ is a normal subgroup of $W_\mathrm{ext}$ and $W_\mathrm{ext} / W_\mathrm{aff} \cong X / \Z\Phi$.
Furthermore, $W_\mathrm{aff}$ is a Coxeter group with simple reflections $S = \{ s_\alpha \mid \alpha \in \Pi \} \sqcup \{ s_0 \}$, where $s_0 = t_{\alpha_\mathrm{h}} s_{\alpha_\mathrm{h}}$, and we write $\ell \colon W_\mathrm{aff} \to \Z_{\geq 0}$ for the length function with respect to $S$.

We consider the $\ell$-dilated dot action of $W_\mathrm{ext}$ on $X_\R$ given by
\[ t_\lambda w \Cdot x = \ell \cdot \lambda + w(x+\rho)-\rho \]
for $\lambda \in X$, $w \in W_\mathrm{fin}$ and $x \in X_\R$.
For $\beta \in \Phi^+$ and $m \in \Z$, the fixed points of the affine reflection $t_{m\beta} s_\beta$ form an affine hyperplane
\[ H_{\beta,m} = \big\{ x \in X_\R \mathrel{\big|} ( x + \rho , \beta^\vee ) = \ell \cdot m \big\} \]
and we call any connected component of $X_\R \setminus \big( \bigcup_{\beta,m} H_{\beta,m}\big)$ an \emph{alcove}.
A weight $\lambda \in X$ is called \emph{$\ell$-singular} if $\lambda \in H_{\beta,m}$ for some $\beta \in \Phi^+$ and $m \in \Z$ and \emph{$\ell$-regular} if $\lambda \in C$ for some alcove $C \subseteq X_\R$.
The group $W_\mathrm{aff}$ acts simply transitively on the set of alcoves, and for every alcove $C \subseteq X_\R$, the closure $\overline{C}$ is a fundamental domain for the $\ell$-dilated dot action of $W_\mathrm{aff}$ on $X_\R$ and $\overline{C} \cap X$ is a fundamental domain for the $\ell$-dilated dot action of $W_\mathrm{aff}$ on $X$.
The \emph{fundamental alcove} is 
\[ C_\mathrm{fund} \coloneqq \{ x \in X_\R \mid 0 < (x+\rho,\beta^\vee) < \ell \text{ for all } \beta\in\Phi^+ \} . \]
Let us write $\Omega = \Stab_{W_\mathrm{ext}}(C_\mathrm{fund})$ for the stabilizer of $C_\mathrm{fund}$ in $W_\mathrm{ext}$.
As $\overline{C}_\mathrm{fund}$ is a fundamental domain for the action of $W_\mathrm{aff}$ on $X_\R$, we have $W_\mathrm{ext} = W_\mathrm{aff} \rtimes \Omega$ and
\[ \Omega = \Stab_{W_\mathrm{ext}}(C_\mathrm{fund}) \cong W_\mathrm{ext} / W_\mathrm{aff} \cong X / \Z \Phi . \]
We write $x\mapsto \omega_x$
for the canonical epimorphism $W_\mathrm{ext} \to \Omega$ with kernel $W_\mathrm{aff}$.

As $W_\mathrm{aff}$ acts simply transitively on the set of alcoves, there is for every element $x \in W_\mathrm{ext}$ a unique element $x^\prime \in W_\mathrm{aff}$ such that $x \Cdot C_\mathrm{fund} = x^\prime \Cdot C_\mathrm{fund}$, and we extend the length function on $W_\mathrm{aff}$ to $W_\mathrm{ext}$ via $\ell(x) \coloneqq \ell(x^\prime)$.
Observe that
\[ \Omega = \{ x \in W_\mathrm{ext} \mid \ell(x) = 0 \} \]
and $\ell(x\omega) = \ell(x)$ for all $x \in W_\mathrm{ext}$ and $\omega \in \Omega$.

An alcove $C \subseteq X_\R$ is called dominant if $(x+\rho , \alpha^\vee) > 0$ for all $x \in C$ and $\alpha \in \Phi^+$.
We define
\[ W_\mathrm{ext}^+ \coloneqq \{ w \in W_\mathrm{ext} \mid w \Cdot C_\mathrm{fund} \text{ is dominant} \} \]
and $W_\mathrm{aff}^+ \coloneqq W_\mathrm{ext}^+ \cap W_\mathrm{aff}$.
It is well-known that every (left or right) $W_\mathrm{fin}$-coset in $W_\mathrm{ext}$ has a unique element of minimal length, and that we have
\[ W_\mathrm{ext}^+ = \{ w \in W_\mathrm{ext} \mid w \text{ has minimal length in } W_\mathrm{fin} w \} . \]

\subsection{Linkage and translation} \label{subsec:linkagetranslation}

For $\lambda \in \overline{C}_\mathrm{fund} \cap X$, the \emph{linkage class}
$\Rep_\lambda(\G)$
of $\lambda$ is the full subcategory of $\Rep(\G)$ whose objects are the $\G$-modules all of whose composition factors are of the form $L(x\Cdot \lambda)$, for some $x\in W_\mathrm{aff}$.
We say that $\Rep_\lambda(\G)$ is \emph{$\ell$-regular} if $\lambda \in C_\mathrm{fund}$ and that $\Rep_\lambda(\G)$ is \emph{$\ell$-singular} if $\lambda \in \overline{C}_\mathrm{fund} \setminus C_\mathrm{fund}$.
Every $\G$-module $M$ admits a decomposition
\[ M = \bigoplus_{ \lambda \in \overline{C}_\mathrm{fund} \cap X } \pr_\lambda M , \]
where $\pr_\lambda M$ denotes the unique largest submodule of $M$ that belongs to $\Rep_\lambda(\G)$; see Sections II.7.1--3 in \cite{Jantzen}.
This gives rise to a decomposition
\[ \Rep(\G) = \bigoplus_{ \lambda \in \overline{C}_\mathrm{fund} \cap X } \Rep_\lambda(\G) \]
with projection functors $\pr_\lambda \colon \Rep(\G) \to \Rep_\lambda(\G)$.
The linkage class $\Rep_0(\G)$ containing the trivial $\G$-module $L(0) \cong \kk$ is called the \emph{principal block}
of $\G$, and we call
\[ \Rep_{\Omega\Cdot0}(\G) \coloneqq \bigoplus_{\lambda\in\Omega\Cdot0} \Rep_\lambda(\G) \]
the \emph{extended principal block}.

For $\lambda , \mu \in \overline{C}_\mathrm{fund} \cap X$, let $\nu \in X^+$ be the unique dominant weight in the $W_\mathrm{fin}$-orbit of $\mu-\lambda$ and define the translation functor from $\Rep_\lambda(\G)$ to $\Rep_\mu(\G)$ via
\[ T_\lambda^\mu = \pr_\mu\big( L(\nu) \otimes - \big) \colon \Rep_\lambda(\G) \longrightarrow \Rep_\mu(\G) . \]
As observed in Sections II.7.6--7 in \cite{Jantzen}, the simple module $L(\nu)$ in the definition of $T_\lambda^\mu$ can be replaced by any $\G$-module of highest weight $\nu$, such as $\nabla(\nu)$, $\Delta(\nu)$ or $T(\nu)$, without changing $T_\lambda^\mu$ (up to a natural isomorphism).
If $\lambda$ and $\mu$ are $\ell$-regular then $T_\lambda^\mu$ is an equivalence with
\[ T_\lambda^\mu L(x \Cdot \lambda) \cong L(x \Cdot \mu) , \quad T_\lambda^\mu \Delta(x \Cdot \lambda) \cong \Delta(x \Cdot \mu) , \quad T_\lambda^\mu \nabla(x \Cdot \lambda) \cong \nabla(x \Cdot \mu) , \quad T_\lambda^\mu T(x \Cdot \lambda) \cong T(x \Cdot \mu) \]
for all $x \in W_\mathrm{aff}^+$, by Propositions II.7.9, II.7.11 and II.7.15 in \cite{Jantzen}.

\subsection{Negligible tilting modules} \label{sec:negligible}

Suppose from now on that $\ell \geq h$, the Coxeter number of $\G$, so that $C_\mathrm{fund} \cap X \neq \varnothing$.
A tilting $\G$-module $T$ is called \emph{negligible} if $\big[ T : T(\nu) \big]_\oplus = 0$ for all $\nu \in C_\mathrm{fund} \cap X$.
By results of H.H.~Andersen and J.~Pardowski \cite{AndersenParadowskiFusionCategories}, the set $\mathcal{N}$ of negligible tilting modules forms a thick tensor ideal in $\Tilt(\G)$, that is, $\mathcal{N}$ is closed under isomorphism, direct sums, retracts, and tensor products with arbitrary tilting $\G$-modules.
The quotient category\footnote{By the quotient category $\Tilt(\G) / \mathcal{N}$, we mean the category whose objects are the tilting $\G$-modules, and where the $\Hom$-space for two tilting $\G$-modules $T$ and $T^\prime$ is the quotient of $\Hom_\G(T,T^\prime)$ by the subspace of homomorphisms that factor through a negligible tilting module.} $\Ver(\G) = \Tilt(\G) / \mathcal{N}$ is a semisimple monoidal category called the \emph{Verlinde category} or the \emph{semisimplification} of $\Rep(\G)$, cf.\ \cite{EtingofOstrikSemisimplification}.
The tilting modules $T(\lambda)$ with $\lambda \in C_\mathrm{fund} \cap X$ form a set of representatives for the isomorphism classes of indecomposable objects in $\Ver(\G)$, and we write
\[ c_{\lambda,\mu}^\nu \coloneqq [ T(\lambda) \otimes T(\mu) : T(\nu) ]_\oplus , \]
for $\lambda,\mu,\nu \in C_\mathrm{fund} \cap X$, for the structure constants of $\Ver(\G)$.
By Proposition II.E.12 in \cite{Jantzen}, we also have
\[ c_{\lambda,\mu}^\nu = \sum_{x\in W_\mathrm{aff}} (-1)^{\ell(x)} \cdot \dim \Delta(\lambda)_{x\Cdot\nu-\mu} . \]
For later use, we need to establish two elementary properties of these structure constants.
We first prove that they are invariant under the action of $\Omega=\Stab_{W_\mathrm{ext}}(C_\mathrm{fund})$ on $C_\mathrm{fund}$, in the following sense:

\begin{Lemma} \label{lem:VerlindeFundamentalGroup}
	Let $\lambda,\mu,\nu\in C_\mathrm{fund} \cap X$ and $\omega\in\Omega$. Then
	\[ \big[ T(\lambda) \otimes T(\omega\Cdot\mu) : T(\omega\Cdot\nu) \big]_\oplus = \big[ T(\lambda) \otimes T(\mu) : T(\nu) \big]_\oplus . \]
	In particular, we have $T(\lambda) \otimes T(\omega\Cdot 0) \cong T(\omega\Cdot\lambda)$ in $\Ver(\G)$.
\end{Lemma}
\begin{proof}
	As conjugation by $\omega$ is an automorphism of $W_\mathrm{aff}$, we have
	\[ [ T(\lambda) \otimes T(\omega\Cdot\mu) : T(\omega\Cdot\nu) ]_\oplus = \sum_{x\in W_\mathrm{aff}} \dim \Delta(\lambda)_{x\omega\Cdot\nu-\omega\Cdot\mu} = \sum_{x\in W_\mathrm{aff}} \dim \Delta(\lambda)_{\omega x\Cdot\nu-\omega\Cdot\mu} . \]
	Writing $\omega = t_\gamma w$ with $\gamma\in X$ and $w\in W_\mathrm{fin}$, it is straightforward to see that
	\[ \omega x\Cdot\nu - \omega\Cdot\mu = w( x\Cdot \nu - \mu ) \]
	and therefore $\dim \Delta(\lambda)_{\omega x\Cdot\nu-\omega\Cdot\mu} = \dim \Delta(\lambda)_{x\Cdot\nu-\mu}$. We conclude that
	\[ [ T(\lambda) \otimes T(\omega\Cdot\mu) : T(\omega\Cdot\nu) ]_\oplus = \sum_{x\in W_\mathrm{aff}} \dim \Delta(\lambda)_{x\Cdot\nu-\mu} = [ T(\lambda) \otimes T(\mu) : T(\nu) ]_\oplus , \]
	as claimed.
\end{proof}

\begin{Lemma} \label{lem:Verlindecoefficientnonzero}
	Let $\lambda,\mu\in C_\mathrm{fund} \cap X$ and denote by $\nu$ be the unique dominant weight in the $W_\mathrm{fin}$-orbit of $w_0\lambda+\mu$. Then $\nu \in C_\mathrm{fund} \cap X$ and $c_{\lambda,\mu}^\nu \neq 0$.
\end{Lemma}
\begin{proof}
	As $-w_0\nu$ is the unique dominant weight in the $W_\mathrm{fin}$-orbit of $-w_0\lambda-\mu$,
	we have
	\[ T(-w_0\lambda) \cong T_\mu^{-w_0\lambda} T(\mu) \cong \pr_{-w_0\lambda}\big( T(-w_0\nu) \otimes T(\mu) \big) , \]
	so $T(-w_0\lambda)$ is a direct summand of $T(-w_0\nu) \otimes T(\mu)$.
	Analogously, we see that
	\[ T(0) \cong T_\lambda^0 T(\lambda) \cong \pr_0\big( T(-w_0\lambda) \otimes T(\lambda) \big) , \]
	whence $T(0)$ is a direct summand of $T(-w_0\lambda) \otimes T(\lambda)$ and of $T(-w_0\nu) \otimes T(\mu) \otimes T(\lambda)$.
	Hence there exists a weight $\nu^\prime \in X^+$ such that $T(\nu^\prime)$ is a direct summand of $T(\mu) \otimes T(\lambda)$ and $T(0)$ is a direct summand of $T(-w_0\nu) \otimes T(\nu^\prime)$.
	Now $T(0)$ is non-negligible, and as the negligible tilting modules form a thick tensor ideal in $\Tilt(\G)$, it follows that $T(\nu^\prime)$ is non-negligible and $\nu^\prime \in C_\mathrm{fund} \cap X$.
	Furthermore, the existence of a non-zero homomorphism from the trivial $\G$-module $T(0)$ to the tensor product $T(-w_0\nu) \otimes T(\nu^\prime) \cong L(-w_0\nu) \otimes L(\nu^\prime)$ implies that $L(\nu^\prime) \cong L(-w_0\nu)^* \cong L(\nu)$ by Schur's lemma.
	We conclude that $\nu=\nu^\prime$, so $T(\nu)$ is a direct summand of $T(\lambda) \otimes T(\mu)$, as claimed.
\end{proof}

\section{Minimal tilting complexes} \label{sec:minimaltilting}

In this section, we recall the theory of minimal tilting complexes explained in \cite{GruberMinimalTilting}.
We write $C^b\big( \Tilt(\G) \big)$ for the category of bounded complexes of tilting $\G$-modules, $K^b\big( \Tilt(\G)\big)$ for the bounded homotopy category of $\Tilt(\G)$ and $D^b\big( \Rep(\G) \big)$ for the bounded derived category of $\Rep(\G)$.
Since $\Rep(\G)$ is a highest weight category, the canonical functor
\[ \mathfrak{T} \colon K^b\big(\mathrm{Tilt}(\G)\big) \longrightarrow D^b\big( \Rep(\G) \big) \]
is an equivalence of triangulated categories.
Thus, for every bounded complex $X$ of $\G$-modules, there is a unique homotopy class of bounded complexes $C$ of tilting $\G$-modules such that $C \cong X$ in the derived category.
Furthermore, as $\Tilt(\G)$ is a Krull--Schmidt category, every homotopy class of bounded complexes of tilting $\G$-modules contains a unique \emph{minimal complex} (up to isomorphism of complexes).
The \emph{minimal tilting complex} $C_\mathrm{min}(X)$ of $X$ is the unique (up to isomorphism) bounded minimal complex of tilting $\G$-modules that is isomorphic to $X$ in the derived category.
The minimal tilting complex $C_\mathrm{min}(M)$ of a $\G$-module $M$ is defined as $C_\mathrm{min}( 0 \to M \to 0 )$, with $M$ in degree zero.
By construction, $C_\mathrm{min}(M)$ is the unique bounded minimal complex of tilting $\G$-modules with
\[ H^i\big( C_\mathrm{min}(M) \big) \cong \begin{cases}
M & \text{if } i=0 , \\ 0 & \text{otherwise} .
\end{cases} \]
We next recall some elementary properties of minimal tilting complexes from \cite{GruberMinimalTilting} and \cite{GruberTensorIdeals}.

\begin{Lemma} \label{lem:minimaltiltingcomplex}
	Let $M$, $M_1$ and $M_2$ be $\G$-modules.
	\begin{enumerate}
		\item If $M$ is a tilting module then $C_\mathrm{min}(M)=M$, viewed as a one-term complex with $M$ in degree~$0$.
		\item We have $C_\mathrm{min}(M_1\oplus M_2)\cong C_\mathrm{min}(M_1) \oplus C_\mathrm{min}(M_2)$ in $C^b\big( \mathrm{Tilt}(\G) \big)$.
		\item If $C$ is a bounded complex of tilting modules with $C \cong M$ in $D^b\big( \Rep(\G) \big)$ then $C_\mathrm{min}(M)$ is the minimal complex of $C$ and there is a split monomorphism $C_\mathrm{min}(M) \to C$ in $C^b\big( \mathrm{Tilt}(\G) \big)$.
		\item $C_\mathrm{min}(M_1 \otimes M_2)$ is the minimal complex of $C_\mathrm{min}(M_1) \otimes C_\mathrm{min}(M_2)$. In particular, there is a split monomorphism $C_\mathrm{min}(M_1 \otimes M_2) \to C_\mathrm{min}(M_1) \otimes C_\mathrm{min}(M_2)$ in $C^b\big(  \Tilt(\G) \big)$.
	\end{enumerate}
\end{Lemma}
\begin{proof}
	Parts (1)--(3) are Remark 2.10 and Lemma 2.12 in \cite{GruberMinimalTilting}, and part (4) is Lemma 1.3 in \cite{GruberTensorIdeals}.
\end{proof}

\begin{Lemma} \label{lem:minimaltiltingcomplexlinkageclass}
	Let $\lambda\in \overline{C}_\mathrm{fund} \cap X$ and let $M$ be a $\G$-module in $\Rep_\lambda(\G)$. Then all terms of $C_\mathrm{min}(M)$ belong to $\Rep_\lambda(\G)$.
\end{Lemma}
\begin{proof}
	As $M$ belongs to $\Rep_\lambda(\G)$ and the projection functor $\mathrm{pr}_\lambda\colon \Rep(\G) \to \Rep_\lambda(\G)$ is exact, we have
	$ M \cong \mathrm{pr}_\lambda M \cong \mathrm{pr}_\lambda C_\mathrm{min}(M) $
	in $D^b\big( \Rep(\G) \big)$.
	By part (3) of Lemma \ref{lem:minimaltiltingcomplex}, $C_\mathrm{min}(M)$ admits a split monomorphism into $\mathrm{pr}_\lambda C_\mathrm{min}(M)$ in $C^b\big( \Tilt(\G) \big)$, and the claim follows.
\end{proof}

Let us introduce an additional piece of notation: For $\G$-modules $M$ and $N$, we write $M \stackrel{\oplus}{\subseteq} N$ if there exists a split monomorphism from $M$ into $N$.
The next result is Lemma 2.17 in \cite{GruberMinimalTilting}.

\begin{Lemma} \label{lem:minimalcomplexSES}
	Let $X \to Y \to Z \to X[1]$ be a distinguished triangle in $D^b\big( \Rep(\G) \big)$.
	Then
	\[ C_\mathrm{min}(Z)_i \stackrel{\oplus}{\subseteq} C_\mathrm{min}(X)_{i+1} \oplus C_\mathrm{min}(Y)_i \]
	for all $i\in\Z$.
	For $\lambda \in X$ and $i\in\Z$, we have 
	\begin{align*}
	\big[ C_\mathrm{min}(X)_{i+1} : & T(\lambda) \big]_\oplus + \big[ C_\mathrm{min}(Y)_i : T(\lambda) \big]_\oplus
	\geq \big[ C_\mathrm{min}(Z)_i : T(\lambda) \big]_\oplus \\
	& \geq  \big[ C_\mathrm{min}(X)_{i+1} : T(\lambda) \big]_\oplus - \big[ C_\mathrm{min}(X)_i : T(\lambda) \big]_\oplus - \big[ C_\mathrm{min}(X)_{i+2} : T(\lambda) \big]_\oplus \\
	& \hspace{2cm} + \big[ C_\mathrm{min}(Y)_i : T(\lambda) \big]_\oplus - \big[ C_\mathrm{min}(Y)_{i-1} : T(\lambda) \big]_\oplus - \big[ C_\mathrm{min}(Y)_{i+1} : T(\lambda) \big]_\oplus .
	\end{align*}
\end{Lemma}

Now we proceed to study the minimal complexes of some specific $\G$-modules.
Let us assume from now on that $\ell \geq h$, the Coxeter number of $\G$, and recall that we write $x \mapsto \omega_x$ for the canonical epimorphism $W_\mathrm{ext} = W_\mathrm{aff} \rtimes \Omega \to \Omega$, where $\Omega=\Stab_{W_\mathrm{ext}}(C_\mathrm{fund})$.

\begin{Proposition} \label{prop:minimalcomplexWeylmodule}
	Let $x\in W_\mathrm{ext}^+$ and $\lambda\in C_\mathrm{fund} \cap X$, and write $C_\mathrm{min}\big( \Delta(x\Cdot \lambda) \big)$ as
	\[ \cdots \xrightarrow{~d_{-2}~} T_{-1} \xrightarrow{~d_{-1}~} T_0 \xrightarrow{~\,d_0\,~} T_1 \xrightarrow{~\,d_1\,~} \cdots . \]
	Then
	\begin{enumerate}
		\item $T_i=0$ for all $i<0$ and all $i>\ell(x)$;
		\item if $\nu \in X^+$ and $i\in\Z$ such that $[T_i:T(\nu)]_\oplus \neq 0$ then $\nu = y\omega_x\Cdot \lambda$ for some $y\in W_\mathrm{aff}^+$ with \[  0 \leq i \leq \ell(x)-\ell(y) ; \]
		\item $T_0 \cong T(x\Cdot \lambda)$ and $T_{\ell(x)} \cong T(\omega_x\Cdot\lambda)$;
		\item $T_i$ is negligible for all $i \neq \ell(x)$.
	\end{enumerate}
\end{Proposition}
\begin{proof}
	For $x^\prime \coloneqq x\omega_x^{-1}$, we have $x^\prime \in W_\mathrm{aff}^+$ and $\ell(x^\prime) = \ell(x)$.
	Hence, after replacing $x$ by $x^\prime$ and $\lambda$ by $\omega_x \Cdot \lambda$, we may (and shall) assume that $x\in W_\mathrm{aff}^+$ and $\omega_x = e$.
	
	We prove the claims by induction on $\ell(x)$.
	If $\ell(x)=0$ then $x=e$ and $\Delta(\lambda) \cong T(\lambda)$, so $\Delta(\lambda)$ has minimal tilting complex $0 \to T(\lambda) \to 0$ and all claims are satisfied.
	Now suppose that $\ell(x)>0$ and that the proposition holds for all $y \in W_\mathrm{aff}^+$ with $\ell(y) < \ell(x)$.
	Then we can choose a simple reflection $s\in S$ with $xs\in W_\mathrm{aff}^+$
	and $xs\Cdot\lambda < x\Cdot\lambda$.
	Let $\mu\in \overline{C}_\mathrm{fund} \cap X$ with $\Stab_{W_\mathrm{aff}}(\mu)=\{e,s\}$, and consider the short exact sequence
	\begin{equation} \label{eq:SESweylmodules}
	 0 \longrightarrow \Delta(x\Cdot \lambda) \longrightarrow T_\mu^\lambda \Delta(x\Cdot\mu) \longrightarrow \Delta(xs\Cdot\lambda) \longrightarrow 0 ,
	\end{equation}
	which is obtained from the short exact sequence in \cite[Proposition II.7.19]{Jantzen} by taking duals.
	Furthermore, let us write
	\[ C_\mathrm{min}\big( T_\mu^\lambda \Delta(x\Cdot\mu) \big) = (\cdots \to A_i \to A_{i+1} \to \cdots) , \qquad  C_\mathrm{min}\big( \Delta(xs\Cdot\lambda) \big) = (\cdots \to B_i \to B_{i+1} \to \cdots) .  \]
	The short exact sequence \eqref{eq:SESweylmodules} gives rise to a distinguished triangle
	\[  T_\mu^\lambda \Delta(x\Cdot\mu) \longrightarrow \Delta(xs\Cdot\lambda) \longrightarrow \Delta(x\Cdot \lambda)[1] \longrightarrow T_\mu^\lambda \Delta(x\Cdot\mu)[1] \]
	in $D^b\big( \Rep(\G) \big)$, and Lemma \ref{lem:minimalcomplexSES} implies that $T_i$ is a direct summand of $C_i \coloneqq A_i \oplus B_{i-1}$ for all $i \in \Z$.
	By the induction hypothesis, we may assume that $B_i=0$ for $i<0$ and $i > \ell(xs) = \ell(x)-1$, that $B_i$ is negligible for $i\neq \ell(x)-1$, that $B_{\ell(x)-1} \cong T(\lambda)$ and that all weights $\nu\in X^+$ with $[ B_i : T(\nu) ]_\oplus \neq 0$ for some $i \in \Z$ are of the form $y\Cdot\lambda$ for some $y\in W_\mathrm{aff}^+$ with $0 \leq i \leq \ell(xs) - \ell(y)$.
	By Proposition 7.11 in \cite{Jantzen}, we have
	\[ \Delta(x\Cdot\mu) = \Delta(xs\Cdot\mu) \cong T_\lambda^\mu \Delta(xs\Cdot\lambda) , \]
	and it follows that $T_\mu^\lambda \Delta(x\Cdot\mu)$ is isomorphic to the complex $T_\mu^\lambda T_\lambda^\mu C_\mathrm{min}\big( \Delta(xs\Cdot\lambda) \big)$ in $D^b\big( \Rep(\G) \big)$.
	Using Lemma \ref{lem:minimaltiltingcomplex}, we conclude that $A_i$ is a direct summand of $T_\mu^\lambda T_\lambda^\mu B_i$ for all $i\in\Z$, and it follows that $A_i=0$ for $i<0$ and $i>\ell(x)-1$.
	Further note that all tilting modules in $\Rep_\mu(\G)$ are negligible because $\mu\notin C_\mathrm{fund}$ and that the translation functor $T_\mu^\lambda$ sends negligible tilting modules to negligible tilting modules, because negligible tilting modules form a thick tensor ideal in $\Tilt(\G)$.
	It follows that the functor $T_\mu^\lambda \circ T_\lambda^\mu$ sends all tilting modules to negligible tilting modules, so $T_\mu^\lambda T_\lambda^\mu B_i$ and $A_i$ are negligible for all $i\in\Z$.
	We conclude that $C_i = A_i \oplus B_{i-1} = 0$ for $i<0$ and $i>\ell(x)$, that $C_i$ is negligible for all $i\neq \ell(x)$ and that
	\[ C_{\ell(x)} \cong A_{\ell(x)} \oplus B_{\ell(x)-1} \cong T(\lambda) . \]
	As $T_i$ is a direct summand of $C_i$ for all $i\in\Z$, this implies that $T_i=0$ for $i<0$ and $i>\ell(x)$ and that $T_i$ is negligible for all $i\neq \ell(x)$.
	Furthermore, Lemma \ref{lem:minimalcomplexSES} yields
	\[ 1 = [ C_{\ell(x)} : T(\lambda) ]_\oplus \geq [ T_{\ell(x)} : T(\lambda) ]_\oplus \geq [ C_{\ell(x)} : T(\lambda) ]_\oplus - [ C_{\ell(x)-1} : T(\lambda) ]_\oplus - [ C_{\ell(x)+1} : T(\lambda) ]_\oplus = 1 \]
	because $C_{\ell(x)-1}$ is negligible and $C_{\ell(x)+1}=0$, and we conclude that $T_{\ell(x)} \cong T(\lambda)$.
	
	Now suppose that $\nu\in X^+$ such that 
	\[ 0 \neq \big[ A_i : T(\nu) \big]_\oplus \leq \big[ T_\mu^\lambda T_\lambda^\mu B_i : T(\nu) \big]_\oplus + \big[ B_i : T(\nu) \big]_\oplus \]
	for some $i\in\Z$.
	If $\big[ B_i : T(\nu) \big]_\oplus \neq 0$ then there is $y \in W_\mathrm{aff}^+$ with $\nu = y \Cdot \lambda$ and $0 \leq i \leq \ell(xs)-\ell(y)$, and if $\big[ T_\mu^\lambda T_\lambda^\mu B_i : T(\nu) \big]_\oplus \neq 0$ then there is $y \in W_\mathrm{aff}^+$ with $0 \leq i \leq \ell(xs)-\ell(y)$ such that $T(\nu)$ is a direct summand of $T_\mu^\lambda T_\lambda^\mu T(y\Cdot\lambda)$.
	By weight considerations, $T_\mu^\lambda T_\lambda^\mu T(y\Cdot\lambda)$ is a direct sum of indecomposable tilting modules of the form $T(z\Cdot\lambda)$ for $z \in W_\mathrm{aff}^+$ with $\ell(z) \leq \ell(y)+1$, and it follows that (in either case) there is $y^\prime \in W_\mathrm{aff}^+$ with $\nu = y^\prime \Cdot \lambda$ and
	\[ 0 \leq i \leq \ell(xs) - \ell(y^\prime) + 1 = \ell(x) - \ell(y^\prime) , \]
	matching (2).
	
	It remains to show that $T_0 \cong T(x\Cdot\lambda)$.
	By Section II.E.4 in \cite{Jantzen}, there is a short exact sequence
	\[ 0 \longrightarrow \Delta(x\Cdot\lambda) \longrightarrow T(x\Cdot\lambda) \longrightarrow M \longrightarrow 0 , \]
	where $M$ is a $\G$-module with a Weyl filtration.
	By Corollary 2.16 in \cite{GruberMinimalTilting}, we have $C_\mathrm{min}(M)_i=0$ for $i<0$, and using Lemma \ref{lem:minimalcomplexSES} as above, we obtain
	\[ T_0 \stackrel{\oplus}{\subseteq} C_\mathrm{min}\big( T(x\Cdot\lambda) \big)_0 \oplus C_\mathrm{min}(M)_{-1} \cong T(x\Cdot\lambda) , \]
	whence $T_0 \cong T(x\Cdot\lambda)$, as required.
\end{proof}

\begin{Proposition} \label{prop:minimalcomplexsimplemodule}
	Let $x\in W_\mathrm{ext}^+$ and $\lambda\in C_\mathrm{fund} \cap X$, and write $C_\mathrm{min}\big( L(x\Cdot \lambda) \big)$ as
	\[ \cdots \xrightarrow{~d_{-2}~} T_{-1} \xrightarrow{~d_{-1}~} T_0 \xrightarrow{~\,d_0\,~} T_1 \xrightarrow{~\,d_1\,~} \cdots . \]
	Then
	\begin{enumerate}
		\item $T_i \cong T_{-i}$ for all $i\in\Z$;
		\item $T_i=0$ for all $i\in\Z$ with $\abs{i}>\ell(x)$;
		\item if $\mu \in X^+$ and $i\in\Z$ with $[T_i:T(\mu)]_\oplus \neq 0$ then $\mu = y\omega_x\Cdot \lambda$ for some $y\in W_\mathrm{aff}^+$ with $\abs{i} \leq \ell(x)-\ell(y)$;
		\item $[T_0 : T(x\Cdot \lambda)]_\oplus = 1$ and $T_{\ell(x)} \cong T_{-\ell(x)} \cong T(\omega_x\Cdot\lambda)$;
		\item $T_{\ell(x)-1} \cong T_{1-\ell(x)}$ is negligible.
	\end{enumerate}
\end{Proposition}
\begin{proof}
	As in the proof of Proposition \ref{prop:minimalcomplexWeylmodule}, we can replace $x$ by $x\omega_x^{-1} \in W_\mathrm{aff}^+$ and $\lambda$ by $\omega_x\Cdot\lambda \in C_\mathrm{fund} \cap X$, so we will henceforth assume that $x \in W_\mathrm{aff}^+$ and $\omega_x = e$.
	By Sections II.2.12 and II.E.6 in \cite{Jantzen}, there is a contravariant duality functor $M \mapsto M^\tau$ on $\Rep(\G)$ (denoted by $M \mapsto \prescript{\tau}{}{M}$ in \cite{Jantzen}) which fixes all simple $\G$-modules and all tilting $\G$-modules.
	(In the quantum case, this functor can be constructed using the involution $\omega$ from Lemma 4.6 in \cite{JantzenLecturesQuantumGroups}.)
	Thus, the complex
	\[ \cdots \xrightarrow{~\,d_1^\tau\,~} T_1^\tau \xrightarrow{~\,d_0^\tau\,~} T_0^\tau \xrightarrow{~d_{-1}^\tau~} T_{-1}^\tau \xrightarrow{~d_{-2}^\tau~} \cdots \]
	is a minimal tilting complex of $L(x\Cdot\lambda)$, and by uniqueness, we have $T_i \cong T_{-i}^\tau \cong T_{-i}$ for all $i$.
	We prove the remaining claims by induction on~$\ell(x)$.
	If $\ell(x)=0$ then $x=e$ and $L(\lambda) \cong T(\lambda)$, so $L(\lambda)$ has minimal tilting complex $0 \to T(\lambda) \to 0$ and all claims are satisfied.
	Now suppose that $\ell(x)>0$ and that the proposition holds for all $y \in W_\mathrm{aff}^+$ with $\ell(y) < \ell(x)$.
	Consider the short exact sequence
	\[ 0 \longrightarrow \rad_\G \Delta(x\Cdot\lambda) \longrightarrow \Delta(x\Cdot\lambda) \longrightarrow L(x\Cdot\lambda) \longrightarrow 0 \]
	from \cite[Section II.2.14]{Jantzen} and the minimal tilting complexes
	\[ C_\mathrm{min}\big( \rad_\G \Delta(x\Cdot\lambda) \big) = (\cdots \to A_i \to A_{i+1} \to \cdots) , \qquad C_\mathrm{min}\big( \Delta(x\Cdot\lambda) \big) = (\cdots \to B_i \to B_{i+1} \to \cdots) , \]
	and observe that $T_i$ is a direct summand of $C_i \coloneqq A_{i+1} \oplus B_i$ for all $i \in \Z$, by Lemma \ref{lem:minimalcomplexSES}.
	By the induction hypothesis and the linkage principle, we may assume that (1)--(5) are satisfied for the minimal tilting complexes of all composition factors of $\rad_\G \Delta(x\Cdot\lambda)$.
	Using Lemma \ref{lem:minimalcomplexSES} and induction on the length of a composition series of $\rad_\G \Delta(x\Cdot \lambda)$, we see that every weight $\mu\in X^+$ with $[A_i:T(\mu)]_\oplus \neq 0$ for some $i\in\Z$ is of the form $y\Cdot\lambda$, for some $y\in W_\mathrm{aff}^+$ with $\abs{i} \leq \ell(x)-\ell(y)-1$.
	In particular, we have $A_i=0$ for all $i\in\Z$ with $\abs{i}\geq \ell(x)$. 
	Now recall from Proposition \ref{prop:minimalcomplexWeylmodule} that $B_i$ is negligible for all $i\neq \ell(x)$, that $B_{\ell(x)}\cong T(\lambda)$ and that every weight $\mu\in X^+$ with $[B_i:T(\mu)]_\oplus\neq 0$ for some $i\in\Z$ is of the form $y\Cdot\lambda$, for some $y\in W_\mathrm{aff}^+$ with $\abs{i} \leq \ell(x)-\ell(y)$.
	As $C_i = A_{i+1} \oplus B_i$ for all $i\in\Z$, we conclude that every weight $\mu\in X^+$ with $[C_i:T(\mu)]_\oplus\neq 0$ for some $i\in\Z$ is of the form $y\Cdot\lambda$, for some $y\in W_\mathrm{aff}^+$ with $\abs{i} \leq \ell(x)-\ell(y)$.
	Furthermore, we have $A_{i+1}=0$ for $i\geq \ell(x)-1$, so $C_{\ell(x)-1} = B_{\ell(x)-1}$ is negligible and $C_{\ell(x)}=B_{\ell(x)}\cong T(\lambda)$. The claims (2), (3) and (5) are now immediate because $T_i$ is a direct summand of $C_i$ for all $i \in \Z$.
	The first part of claim (4) follows from Lemma \ref{lem:minimalcomplexSES} because
	\[ [ C_0 : T(x\Cdot\lambda) ]_\oplus = [ A_1 : T(x\Cdot\lambda) ]_\oplus + [ B_0 : T(x\Cdot\lambda) ]_\oplus = 1 \] and
	$ [ C_{-1} : T(x\Cdot\lambda) ]_\oplus = 0 = [ C_1 : T(x\Cdot\lambda) ]_\oplus $,
	and therefore
	\[ 1 = [ C_0 : T(x\Cdot\lambda) ]_\oplus \geq [ T_0 : T(x\Cdot\lambda) ]_\oplus \geq [ C_0 : T(x\Cdot\lambda) ]_\oplus - [ C_{-1} : T(x\Cdot\lambda) ]_\oplus - [ C_1 : T(x\Cdot\lambda) ]_\oplus = 1 . \]
	Analogously, we have $C_{\ell(x)} \cong T(\lambda)$ and
	$ [ C_{\ell(x)-1} : T(\lambda) ]_\oplus = 0 = [ C_{\ell(x)+1} : T(\lambda) ]_\oplus $
	because $C_{\ell(x)-1}$ is negligible and $C_{\ell(x)+1}=0$.
	Using Lemma \ref{lem:minimalcomplexSES} again, it follows that $T_{\ell(x)} \cong T(\lambda)$.
\end{proof}

\begin{Remark}
	In the quantum case, Propositions \ref{prop:minimalcomplexWeylmodule} and \ref{prop:minimalcomplexsimplemodule} can also be derived from Theorem 4.8 in \cite{GruberMinimalTilting}, using combinatorial properties of Kazhdan-Lusztig polynomials.
\end{Remark}

\section{The ideal of singular \texorpdfstring{$\G$}{G}-modules} \label{sec:singularmodules}

For the rest of the article, we assume that $\ell \geq h$, the Coexeter number of $\G$.
Let $\mathcal{I}$ be a thick tensor ideal in $\Tilt(\G)$.
In \cite[Lemma 2.3]{GruberTensorIdeals}, it is shown that the set of $\G$-modules
\[ \langle \mathcal{I} \rangle \coloneqq \big\{ M \in \Rep(\G) \mathop{\big|} \text{all terms of } C_\mathrm{min}(M) \text{ belong to } \mathcal{I} \big\} \]
is a thick tensor ideal with the \emph{2/3-property}:
For any short exact sequence
$ 0 \to A \to B \to C \to 0 $
of $\G$-modules such that two of the $\G$-modules $A$, $B$ and $C$ belong to $\langle \mathcal{I} \rangle$, the third also belongs to $\langle \mathcal{I} \rangle$.
Furthermore, by Lemma 2.4 in \cite{GruberTensorIdeals}, $\langle \mathcal{I} \rangle$ is the smallest thick tensor ideal with the 2/3-property in $\Rep(\G)$ that contains $\mathcal{I}$.
Recall from Subsection \ref{sec:negligible} that we write $\mathcal{N}$ for the thick tensor ideal of negligible tilting modules.

\begin{Definition}
	We call $\langle \mathcal{N} \rangle$ the ideal of \emph{singular $\G$-modules}
	and say that a $\G$-module is \emph{regular}
	if it does not belong to $\langle \mathcal{N} \rangle$.
	We refer to the quotient category $\regRep(\G) \coloneqq \Rep(\G) / \langle \mathcal{N} \rangle$ as the \emph{regular quotient}
	of $\Rep(\G)$ and write $q \colon \Rep(\G) \to \regRep(\G)$ for the quotient functor.
\end{Definition}

The quotient category $\regRep(\G)$ has the same objects as $\Rep(\G)$, but for two $\G$-modules $M$ and $N$, the space of homomorphisms from $M$ to $N$ in $\regRep(\G)$ is the quotient of $\Hom_\G(M,N)$ by the space of homomorphisms that factor through a singular $\G$-module.
Since the set $\langle \mathcal{N} \rangle$ of singular $\G$-modules is closed under retracts, a $\G$-module $M$ is regular if and only if $q(M)$ is non-zero in $\regRep(\G)$.
Furthermore, $\regRep(\G)$ inherits the Krull--Schmidt property from $\Rep(\G)$.

We first prove two results that justify our terminology.

\begin{Lemma} \label{lem:singularidealsingularlinkageclasses}
	The ideal $\langle \mathcal{N} \rangle$ of singular $\G$-modules is the smallest thick tensor ideal in $\Rep(\G)$ with the 2/3-property that contains all $\ell$-singular linkage classes.
\end{Lemma}
\begin{proof}
	Recall that a linkage class $\Rep_\mu(\G)$ is called $\ell$-singular if $\mu \in \overline{C}_\mathrm{fund} \setminus C_\mathrm{fund}$.
	For a $\G$-module $M$ in an $\ell$-singular linkage class $\Rep_\mu(\G)$, all terms of the minimal complex $C_\mathrm{min}(M)$ are negligible because they belong to $\Rep_\mu(\G)$ by Lemma \ref{lem:minimaltiltingcomplexlinkageclass}, so $M \in \langle \mathcal{N} \rangle$.
	
	Now let $\mathcal{I}$ be a thick tensor ideal with the 2/3-property that contains all $\ell$-singular linkage classes.
	In order to show that $\mathcal{I}$ contains $\langle \mathcal{N} \rangle$, it suffices to verify that $\mathcal{I}$ contains $\mathcal{N}$, since $\langle \mathcal{N} \rangle$ is the smallest thick tensor ideal with the 2/3-property in $\Rep(\G)$ that contains $\mathcal{N}$.
	All indecomposable tilting modules of $\ell$-singular highest weight belong to $\mathcal{I}$ by assumption, so now consider a negligible tilting module $T(x\Cdot \lambda)$ of $\ell$-regular highest weight, where $\lambda\in C_\mathrm{fund} \cap X$ and $x \in W_\mathrm{aff}^+$ with $x \neq e$.
	Let $s\in S$ be a simple reflection with
	$xs\in W_\mathrm{aff}^+$
	and $xs\Cdot\lambda<x\Cdot\lambda$.
	We can choose a weight $\mu \in \overline{C}_\mathrm{fund} \cap X$ with $\Stab_{W_\mathrm{aff}}(\mu) = \{ e , s \}$,
	and then by Section II.E.11 in \cite{Jantzen}, we have $T_\mu^\lambda T(x\Cdot\mu) \cong T(x\Cdot\lambda)$.
	As $T(x\Cdot\mu)$ belongs to $\mathcal{I}$ and as $T_\mu^\lambda T(x\Cdot\mu)$ is a direct summand of $T(x\Cdot\mu) \otimes T(\nu)$, for $\nu$ the unique dominant weight in the $W_\mathrm{fin}$-orbit of $\lambda-\mu$, we conclude that $T(x\Cdot\lambda)$ belongs to $\mathcal{I}$, as required.
\end{proof}

\begin{Lemma} \label{lem:regularquotientregularweight}
	For $\lambda\in X^+$, the following are equivalent:
	\begin{enumerate}
		\item $\Delta(\lambda)$ is regular;
		\item $L(\lambda)$ is regular;
		\item $\lambda$ is $\ell$-regular.
	\end{enumerate}
\end{Lemma}
\begin{proof}
	Suppose first that $\lambda$ is $\ell$-regular and write $\lambda= x \Cdot \lambda^\prime$ for some $x\in W_\mathrm{aff}^+$ and $\lambda^\prime \in C_\mathrm{fund} \cap X$.
	By Propositions \ref{prop:minimalcomplexWeylmodule} and \ref{prop:minimalcomplexsimplemodule}, the minimal tilting complexes of both $\Delta(\lambda)$ and $L(\lambda)$ have the non-negligible tilting module $T(\lambda^\prime)$ as their term in degree $\ell(x)$, and it follows that $\Delta(\lambda) \notin \langle \mathcal{N} \rangle$ and $L(\lambda) \notin \langle \mathcal{N} \rangle$.
	Conversely, if $\lambda$ is $\ell$-singular then the linkage class containing $\Delta(\lambda)$ and $L(\lambda)$ is contained in $\langle \mathcal{N} \rangle$ by Lemma \ref{lem:singularidealsingularlinkageclasses}, and it follows that $\Delta(\lambda) \in \langle \mathcal{N} \rangle$ and $L(\lambda)  \in \langle \mathcal{N} \rangle$.
\end{proof}

Our next goal is to prove two results that we consider as a `linkage principle' and a `translation principle' for tensor products.
(See Remark \ref{rem:translationtensorproduct} below for an explanation of this terminology.)
The first one (Corollary \ref{cor:principalblockregularquotient}) asserts that the principal block (and the extended principal block) are closed under tensor products in the regular quotient.
The second one (Theorem \ref{thm:translationtensorquotient}) shows that the Krull--Schmidt decomposition of any tensor product in $\regRep(\G)$ can be determined by looking at the Krull--Schmidt decomposition of (the projection to $\Rep_0(\G)$ of) a tensor product of $\G$-modules in $\Rep_0(\G)$ and that the multiplicities of indecomposable direct summands are governed by the Verlinde category.
Our main tool for proving these results will be the following lemma.

\begin{Lemma} \label{lem:linkageclassesregularquotient}
	Let $\lambda\in C_\mathrm{fund} \cap X$ and $\omega\in \Omega$. For $\G$-modules $M$ and $N$ in the linkage classes $\Rep_\lambda(\G)$ and $\Rep_{\omega\Cdot 0}(\G)$, respectively, the canonical embedding
	\[ \pr_{\omega\Cdot\lambda}\big( M \otimes N \big) \longrightarrow M \otimes N \]
	and the canonical projection
	\[ M \otimes N \longrightarrow \pr_{\omega\Cdot\lambda}\big( M \otimes N \big)\]
	descend to isomorphisms in $\regRep(\G)$.
\end{Lemma}
\begin{proof}
	By the linkage principle, we have
	\[ M \otimes N \cong \bigoplus_{\nu\in \overline{C}_\mathrm{fund} \cap X} \pr_\nu\big(M \otimes N\big) , \]
	and the lemma is equivalent to the statement that $\pr_\nu\big(M \otimes N\big) \cong 0$ in the regular quotient $\regRep(\G)$, for all weights $\nu\in \overline{C}_\mathrm{fund} \cap X$ with $\nu \neq \omega\Cdot\lambda$.
	Observe that all terms of $C_\mathrm{min}(M)$ belong to $\Rep_\lambda(\G)$ and all terms of $C_\mathrm{min}(N)$ belong to $\Rep_{\omega\Cdot0}(\G)$ by Lemma \ref{lem:minimaltiltingcomplexlinkageclass}.
	As $C_\mathrm{min}\big( \pr_\nu(M \otimes N) \big)$ admits a split monomorphism into the complex $\pr_\nu\big( C_\mathrm{min}(M) \otimes C_\mathrm{min}(N) \big)$ by part (4) of Lemma \ref{lem:minimaltiltingcomplex}, it suffices to prove that $\pr_\nu\big( T(x\Cdot\lambda) \otimes T(y\omega\Cdot0) \big)$ is negligible for all $x,y\in W_\mathrm{aff}^+$ and $\nu \in \overline{C}_\mathrm{fund} \cap X$ with $\nu \neq \omega\Cdot\lambda$.
	If $x \neq e$ or $y \neq e$ then $T(x\Cdot\lambda) \otimes T(y\omega\Cdot0)$ is negligible, because the negligible tilting modules form a thick tensor ideal in $\Tilt(\G)$.
	For $x=y=e$, we have $T(\lambda) \otimes T(\omega\Cdot 0) \cong T(\omega\Cdot\lambda)$ in the Verlinde category by Lemma \ref{lem:VerlindeFundamentalGroup}, and it follows that $\pr_\nu\big( T(\lambda) \otimes T(\omega\Cdot0) \big)$ is negligible for all $\nu \in \overline{C}_\mathrm{fund} \cap X$ with $\nu \neq \omega\Cdot\lambda$, as required.
\end{proof}

For $\lambda \in C_\mathrm{fund}\cap X$, let us write $\regRep_\lambda(\G)$ for the essential image of the linkage class $\Rep_\lambda(\G)$ under the quotient functor $q \colon \Rep(\G) \to \regRep(\G)$, i.e.~the full subcategory of $\regRep(\G)$ whose objects are the $\G$-modules that are isomorphic to a $\G$-module in $\Rep_\lambda(\G)$, when considered as objects in $\regRep(\G)$.
We also write $\regRep_{\Omega\Cdot0}(\G)$ for the essential image of the extended principal block $\Rep_{\Omega\Cdot0}(\G)$ in $\regRep(\G)$.
As a consequence of Lemma \ref{lem:linkageclassesregularquotient}, we obtain our `linkage principle' for tensor products.

\begin{Corollary} \label{cor:principalblockregularquotient}
	The subcategories $\regRep_0(\G)$ and $\regRep_{\Omega\Cdot0}(\G)$ are closed under tensor products. 
\end{Corollary}
\begin{proof}
	For $\omega,\omega^\prime \in \Omega$ and $\G$-modules $M$ and $N$ in the linkage classes of $\omega\Cdot0$ and $\omega^\prime\Cdot0$, respectively, we have $M \otimes N \cong \pr_{\omega\omega^\prime\Cdot0}( M \otimes N )$ in $\regRep(\G)$ by Lemma \ref{lem:linkageclassesregularquotient}, so $M \otimes N$ belongs to $\regRep_{\omega\omega^\prime\Cdot0}(\G)$. The claim about $\regRep_0(\G)$ follows by setting $\omega=\omega^\prime=e$.
\end{proof}

As a further consequence of Lemma \ref{lem:linkageclassesregularquotient}, we prove that a translation functor with source in the extended principal block descends in the regular quotient to tensoring with a tilting module.

\begin{Corollary} \label{cor:translationquotient}
	Let $\lambda\in C_\mathrm{fund}\cap X$ and $\omega\in \Omega$. Then $\lambda$ is the unique dominant weight in the $W_\mathrm{fin}$-orbit of $\omega\Cdot\lambda - \omega\Cdot0$, and the canonical natural transformations
	\[ T_{\omega\Cdot0}^{\omega\Cdot\lambda} = \pr_{\omega\Cdot\lambda}\big( T(\lambda) \otimes - \big) \hspace{.2cm} \Longrightarrow \hspace{.2cm} \big( T(\lambda) \otimes - \big) \hspace{.2cm} \Longrightarrow \hspace{.2cm} \pr_{\omega\Cdot\lambda}\big( T(\lambda) \otimes - \big) = T_{\omega\Cdot0}^{\omega\Cdot\lambda} \]
	of functors from $\Rep_{\omega\Cdot0}(\G)$ to $\Rep(\G)$ give rise to an isomorphism of functors
	\[ q \circ T_{\omega\Cdot0}^{\omega\Cdot\lambda} \cong q \circ \big( T(\lambda) \otimes - \big) . \]
\end{Corollary}
\begin{proof}
	Writing $\omega=t_\gamma w$ with $\gamma\in X$ and $w \in W_\mathrm{fin}$, it is straightforward to see that $\omega\Cdot\lambda - \omega\Cdot0 = w(\lambda)$, so $\lambda$ is indeed the unique dominant weight in the $W_\mathrm{fin}$-orbit of $\omega\Cdot\lambda - \omega\Cdot0$.
	By Lemma \ref{lem:linkageclassesregularquotient}, the component at a $\G$-module $N$ in $\Rep_{\omega\Cdot0}(\G)$ of either of the two natural transformations descends to an isomorphism in $\regRep(\G)$, and the claim follows.
\end{proof}

We are now ready to establish our `translation principle' for tensor products.

\begin{Theorem} \label{thm:translationtensorquotient}
	For $\lambda,\mu\in C_\mathrm{fund} \cap X$ and $\omega,\omega^\prime \in \Omega$, there is a natural transformation of bifunctors
	\[ \Psi \colon \big( T_{\omega\Cdot0}^{\omega\Cdot\lambda} - \big) \otimes \big( T_{\omega^\prime\Cdot0}^{\omega^\prime\Cdot\mu} - \big) \Longrightarrow \bigoplus_{\nu \in C_\mathrm{fund} \cap X} \big( T_{\omega\omega^\prime\Cdot0}^{\omega\omega^\prime\Cdot\nu} \circ \pr_{\omega\omega^\prime\Cdot0}( - \otimes - ) \big)^{\oplus c_{\lambda,\mu}^\nu} \]
	from $\Rep_{\omega\Cdot0}(\G) \times \Rep_{\omega^\prime\Cdot0}(\G)$ to $\Rep(\G)$, where $c_{\lambda,\mu}^\nu = [ T(\lambda) \otimes T(\mu) : T(\nu) ]_\oplus$,
	such that $q \Psi$ is an isomorphism of bifunctors.
\end{Theorem}
\begin{proof}
	We construct the natural transformation in several steps.
	\begin{enumerate}
		\item By Corollary \ref{cor:translationquotient}, the natural embedding
		\[ \big( T_{\omega\Cdot0}^{\omega\Cdot\lambda} - \big) \otimes \big( T_{\omega^\prime\Cdot0}^{\omega^\prime\Cdot\mu} - \big) = \pr_{\omega\Cdot\lambda}\big( T(\lambda) \otimes - \big) \otimes \pr_{\omega^\prime\Cdot\mu}\big( T(\mu) \otimes - \big) \Longrightarrow \big( T(\lambda) \otimes - \big) \otimes \big( T(\mu) \otimes - \big) \]
		induces an isomorphism of functors upon passage to the regular quotient $\regRep(\G)$.
		\item The braiding on $\Rep(\G)$ gives rise to a natural isomorphism
		\[ \big( T(\lambda) \otimes - \big) \otimes \big( T(\mu) \otimes - \big) \cong \big( T(\lambda) \otimes T(\mu) \big) \otimes ( - \otimes - ) . \]
		\item The canonical projection to the linkage class of $\omega\omega^\prime\Cdot0$ gives rise to a natural transformation
		\[ ( - \otimes -) \Longrightarrow \pr_{\omega\omega^\prime\Cdot0}(-\otimes-) \]
		of bifunctors from $\Rep_{\omega\Cdot0}(\G) \times \Rep_{\omega^\prime\Cdot0}(\G)$ to $\Rep(\G)$, which descends to a natural isomorphism in $\regRep(\G)$ by Lemma \ref{lem:linkageclassesregularquotient}. Tensoring with $T(\lambda) \otimes T(\mu)$ yields a natural transformation
		\[ \big( T(\lambda) \otimes T(\mu) \big) \otimes ( - \otimes -) \Longrightarrow \big( T(\lambda) \otimes T(\mu) \big) \otimes \pr_{\omega\omega^\prime\Cdot0}(-\otimes-) , \]
		which again descends to a natural isomorphism in $\regRep(\G)$.
		\item The tensor product $T(\lambda) \otimes T(\mu)$ can be decomposed as a direct sum
		\[ T(\lambda) \otimes T(\mu) \cong N \oplus \bigoplus_{\nu \in C_\mathrm{fund} \cap X} T(\nu)^{\oplus c_{\lambda,\mu}^\nu} , \]
		where $N$ is a negligible tilting module.
		This decomposition gives rise to a natural isomorphism
		\[ \big( T(\lambda) \otimes T(\mu) \big) \otimes \pr_{\omega\omega^\prime\Cdot0}( - \otimes - ) \cong \big( N \otimes \pr_{\omega\omega^\prime\Cdot0}( - \otimes - ) \big) \oplus \bigoplus_{ \nu \in C_\mathrm{fund} \cap X } \big( T(\nu) \otimes \pr_{\omega\omega^\prime\Cdot0}( - \otimes - ) \big)^{\oplus c_{\lambda,\mu}^\nu} . \]
		As $N$ is negligible, the essential image of the bifunctor $N \otimes \pr_{\omega\omega^\prime\Cdot0}( - \otimes - )$ is contained in $\langle \mathcal{N} \rangle$, and it follows that $q \circ \big( N \otimes \pr_{\omega\omega^\prime\Cdot0}( - \otimes - ) \big) = 0$.
		Therefore, the projection onto the non-negligible part gives rise to a natural transformation
		\[ \big( T(\lambda) \otimes T(\mu) \big) \otimes \pr_{\omega\omega^\prime\Cdot0}( - \otimes - ) \Longrightarrow \bigoplus_{ \nu \in C_\mathrm{fund} \cap X } \big( T(\nu) \otimes \pr_{\omega\omega^\prime\Cdot0}( - \otimes - ) \big)^{\oplus c_{\lambda,\mu}^\nu} , \]
		which descends to an isomorphism of functors in $\regRep(\G)$.
		\item Again by Corollary \ref{cor:translationquotient}, the canonical natural transformation
		\[ \bigoplus_{ \nu \in C_\mathrm{fund} \cap X } \big( T(\nu) \otimes \pr_{\omega\omega^\prime\Cdot0}( - \otimes - ) \big)^{\oplus c_{\lambda,\mu}^\nu} \Longrightarrow \bigoplus_{ \nu \in C_\mathrm{fund} \cap X } \big( T_{\omega\omega^\prime\Cdot0}^{\omega\omega^\prime\Cdot\nu} \circ \pr_{\omega\omega^\prime\Cdot0}( - \otimes - ) \big)^{\oplus c_{\lambda,\mu}^\nu} \]
		induces an isomorphism of functors upon passage to the regular quotient.
	\end{enumerate}
	All of the natural transformations in (1)--(5) give rise to natural isomorphisms upon passage to the regular quotient $\regRep(\G)$.
	Therefore, their composition is a natural transformation
	\[ \Psi \colon \big( T_{\omega\Cdot0}^{\omega\Cdot\lambda} - \big) \otimes \big( T_{\omega^\prime\Cdot0}^{\omega^\prime\Cdot\mu} - \big) \Longrightarrow \bigoplus_{\nu \in C_\mathrm{fund} \cap X} \big( T_{\omega\omega^\prime\Cdot0}^{\omega\omega^\prime\Cdot\nu} \circ \pr_{\omega\omega^\prime\Cdot0}( - \otimes - ) \big)^{\oplus c_{\lambda,\mu}^\nu} \]
	such that $q \Psi$ is a natural isomorphism.
\end{proof}

\begin{Remark}
	The statement of Theorem \ref{thm:translationtensorquotient} becomes more readable (but also slightly less general) if we set $\omega = \omega^\prime = e$:
	For $\lambda,\mu\in C_\mathrm{fund} \cap X$, there is a natural transformation of bifunctors
	\[ \Psi \colon \big( T_0^\lambda - \big) \otimes \big( T_0^\mu - \big) \Longrightarrow \bigoplus_{\nu \in C_\mathrm{fund} \cap X} \big( T_0^\nu \circ \pr_0( - \otimes - ) \big)^{\oplus c_{\lambda,\mu}^\nu} \]
	from $\Rep_0(\G) \times \Rep_0(\G)$ to $\Rep(\G)$, such that $q \Psi$ is an isomorphism of bifunctors.
	Taking the action of $\Omega$ into account complicates our notation here, but it will be useful in applications.
\end{Remark}

\begin{Remark} \label{rem:translationtensorproduct}
	Let us briefly explain why we think of Corollary \ref{cor:principalblockregularquotient} and Theorem \ref{thm:translationtensorquotient} as a `linkage principle' and a `translation principle' for tensor products.
	The usual linkage principle asserts that the category $\Rep(\G)$ decomposes into linkage classes, and the usual translation principle establishes equivalences between the different $\ell$-regular linkage classes.
	Thus, many questions about the structure of the category $\Rep(\G)$ can be reduced to questions about the principal block $\Rep_0(\G)$.
	However, this strategy fails for two reasons when one tries to take the monoidal structure of $\Rep(\G)$ into account. Firstly, the principal block is not closed under tensor products. In fact, the tensor product of two $\G$-modules in $\Rep_0(\G)$ can have non-zero indecomposable direct summands in many different linkage classes, including $\ell$-singular ones. Secondly, it is a priori not clear how structural information about tensor products of $\G$-modules in the principal block can be used to deduce (precise) structural information about tensor products of $\G$-modules in arbitrary $\ell$-regular linkage classes.
	
	The preceding results show that both of these obstacles can be partially resolved by passing to the regular quotient. Indeed, Corollary \ref{cor:principalblockregularquotient} tells us that the essential image $\regRep_0(\G)$ of the principal block in the regular quotient is closed under tensor products; hence, the decomposition of $\Rep(\G)$ into linkage classes is, to some extent, compatible with the monoidal strucure of $\Rep(\G)$.
	Furthermore, Theorem \ref{thm:translationtensorquotient} enables us to describe (the regular parts of) tensor products of $\G$-modules in arbitrary $\ell$-regular linkage classes, once we know the structure of (the components in $\Rep_0(\G)$ of) tensor products of $\G$-modules in $\Rep_0(\G)$.
	The reader should note, however, that all information about singular direct summands is lost in the process.
\end{Remark}

In the following, we present a second approach to the `linkage principle' and the `translation principle' for tensor products, which largely bypasses the quotient category $\regRep(\G)$, but also loses the functoriality of Theorem \ref{thm:translationtensorquotient}.
When studying tensor product of specific $\G$-modules, rather than categorical properties of $\Rep(\G)$, this second approach will turn out to be more convenient.

\begin{Definition}
	For a $\G$-module $M$, we write $M \cong M_\mathrm{sing} \oplus M_\mathrm{reg}$, where for a fixed Krull--Schmidt decomposition of $M$, we define $M_\mathrm{sing}$
	to be the direct sum of the singular indecomposable direct summands of $M$ and $M_\mathrm{reg}$
	to be the direct sum of the regular indecomposable direct summands of $M$.
	We call $M_\mathrm{sing}$ the \emph{singular part}
	of $M$ and $M_\mathrm{reg}$ the \emph{regular part}
	of~$M$.
\end{Definition}

Note that the decomposition $M \cong M_\mathrm{sing} \oplus M_\mathrm{reg}$ in the previous definition is neither canonical nor functorial.
Nevertheless, the singular part and the regular part are uniquely determined up to isomorphism by the Krull--Schmidt decomposition of $M$.

\begin{Lemma} \label{lem:regularpartdirectsumtensorproduct}
	For $\G$-modules $M$ and $N$, we have
	\[ M_\mathrm{reg} \oplus N_\mathrm{reg} \cong (M \oplus N)_\mathrm{reg} \qquad \text{and} \qquad (M \otimes N)_\mathrm{reg} \cong \big( M_\mathrm{reg} \otimes N_\mathrm{reg} \big)_\mathrm{reg} . \]
\end{Lemma}
\begin{proof}
	The first isomorphism is straightforward to see from the definition.
	The second one follows from the direct sum decomposition
	\begin{multline*}
	M \otimes N \cong ( M_\mathrm{reg} \oplus M_\mathrm{sing} ) \otimes ( N_\mathrm{reg} \oplus N_\mathrm{sing} )
	\\ \cong ( M_\mathrm{reg} \otimes N_\mathrm{reg} ) \oplus ( M_\mathrm{reg} \otimes N_\mathrm{sing} ) \oplus ( M_\mathrm{sing} \otimes N_\mathrm{reg} ) \oplus ( M_\mathrm{sing} \otimes N_\mathrm{sing} )
	\end{multline*}
	and the fact that singular $\G$-modules form a thick tensor ideal.
\end{proof}

The following lemma can be seen as another version of the `linkage principle' for tensor products.

\begin{Lemma} \label{lem:regularpartlinkageclass}
	Let $\lambda\in C_\mathrm{fund}\cap X$ and $\omega\in \Omega$, and let $M$ and $N$ be $\G$-modules such that $M_\mathrm{reg}$ belongs to $\Rep_\lambda(\G)$ and $N_\mathrm{reg}$ belongs to $\Rep_{\omega\Cdot0}(\G)$.
	Then $(M \otimes N)_\mathrm{reg}$ belongs to $\Rep_{\omega\Cdot\lambda}(\G)$.
\end{Lemma}
\begin{proof}
	By Lemma \ref{lem:regularpartdirectsumtensorproduct} and the linkage principle, we have
	\[ (M \otimes N)_\mathrm{reg} = \big( M_\mathrm{reg} \otimes N_\mathrm{reg} \big)_\mathrm{reg} \cong \bigoplus_{\nu \in \overline{C}_\mathrm{fund} \cap X} \big( \pr_\nu( M_\mathrm{reg} \otimes N_\mathrm{reg} ) \big)_\mathrm{reg} , \]
	and it suffices to show that $\pr_\nu( M_\mathrm{reg} \otimes N_\mathrm{reg} )$ is singular for all $\nu\in \overline{C}_\mathrm{fund} \cap X$ with $\nu \neq \omega\Cdot\lambda$.
	This was already observed in the proof of Lemma \ref{lem:linkageclassesregularquotient}, for arbitrary $\G$-modules in the linkage classes $\Rep_\lambda(\G)$ and $\Rep_{\omega\Cdot0}(\G)$.
\end{proof}

Next we give a reformulation of Corollary \ref{cor:translationquotient} in terms of regular parts of $\G$-modules.

\begin{Corollary} \label{cor:regularparttranslation}
	Let $\lambda\in C_\mathrm{fund} \cap X$ and $\omega\in\Omega$, and let $M$ be a $\G$-module in $\Rep_{\omega\Cdot0}(\G)$.
	Then
	\[ \big( T(\lambda) \otimes M \big)_\mathrm{reg} \cong \big( T_{\omega\Cdot0}^{\omega\Cdot\lambda} M \big)_\mathrm{reg} . \]
\end{Corollary}
\begin{proof}
	Recall from Corollary \ref{cor:translationquotient} that $\lambda$ is the unique dominant weight in the $W_\mathrm{fin}$-orbit of $\omega\Cdot\lambda - \omega\Cdot0$, so $T_{\omega\Cdot0}^{\omega\Cdot\lambda} = \pr_{\omega\Cdot\lambda}\big( T(\lambda) \otimes - \big)$. Using Lemma \ref{lem:regularpartlinkageclass}, we obtain
	\[ \big( T(\lambda) \otimes M \big)_{\mathrm{reg}} \cong \pr_{\omega\Cdot\lambda}\Big( \big( T(\lambda) \otimes M \big)_{\mathrm{reg}} \Big) \cong \Big( \pr_{\omega\Cdot\lambda}\big( T(\lambda) \otimes M \big) \Big)_{\mathrm{reg}} \cong \big( T_{\omega\Cdot0}^{\omega\Cdot\lambda} M \big)_\mathrm{reg} , \] 
	as required.
\end{proof}

The following result is a non-functorial version of the `translation principle' for tensor products from Theorem \ref{thm:translationtensorquotient}.

\begin{Theorem} \label{thm:regularparttensorproduct}
	Let $\lambda,\mu \in C_\mathrm{fund} \cap X$ and let $M$ and $N$ be $\G$-modules in the linkage classes $\Rep_{\omega\Cdot0}(\G)$ and $\Rep_{\omega^\prime\Cdot0}(\G)$, respectively, for certain $\omega,\omega^\prime\in\Omega$.
	Then
	\[ \big( T_{\omega\Cdot0}^{\omega\Cdot\lambda} M \otimes T_{\omega^\prime\Cdot0}^{\omega^\prime\Cdot\mu} N \big)_\mathrm{reg} \cong \bigoplus_{\nu\in C_\mathrm{fund} \cap X} \big( T_{\omega\omega^\prime\Cdot0}^{\omega\omega^\prime\Cdot\nu} (M \otimes N)_\mathrm{reg} \big)^{\oplus c_{\lambda,\mu}^\nu}  \]
\end{Theorem}
\begin{proof}
	By Lemma \ref{lem:regularpartdirectsumtensorproduct} and Corollary \ref{cor:regularparttranslation}, we have
	\begin{align*}
	\big( T_{\omega\Cdot0}^{\omega\Cdot\lambda} M \otimes T_{\omega^\prime\Cdot0}^{\omega^\prime\Cdot\mu} N \big)_\mathrm{reg} & \cong \Big( \big( T_{\omega\Cdot0}^{\omega\Cdot\lambda} M \big)_\mathrm{reg} \otimes \big( T_{\omega^\prime\Cdot0}^{\omega^\prime\Cdot\mu} N \big)_\mathrm{reg} \Big)_\mathrm{reg} \\
	& \cong \Big( \big( T(\lambda) \otimes M \big)_\mathrm{reg} \otimes \big( T(\mu) \otimes N \big)_\mathrm{reg} \Big)_\mathrm{reg} \\
	& \cong \big( T(\lambda) \otimes M \otimes T(\mu) \otimes N \big)_\mathrm{reg} \\
	& \cong \Big( \big( T(\lambda) \otimes T(\mu) \big)_\mathrm{reg} \otimes \big( M \otimes N \big)_\mathrm{reg} \Big)_\mathrm{reg} .
	\end{align*}
	Now
	\[ \big( T(\lambda) \otimes T(\mu) \big)_\mathrm{reg} \cong \bigoplus_{\nu\in C_\mathrm{fund} \cap X} T(\nu)^{\oplus c_{\lambda,\mu}^\nu}  \]
	and $(M \otimes N)_\mathrm{reg}$ belongs to the linkage class $\Rep_{\omega\omega^\prime\Cdot0}(\G)$ by Lemma \ref{lem:regularpartlinkageclass}. Again using Corollary \ref{cor:regularparttranslation}, we obtain
	\[ \big( T(\nu) \otimes (M \otimes N)_\mathrm{reg} \big)_\mathrm{reg} \cong \big( T_{\omega\omega^\prime\Cdot0}^{\omega\omega^\prime\Cdot\nu} (M \otimes N)_\mathrm{reg} \big)_\mathrm{reg} \cong T_{\omega\omega^\prime\Cdot0}^{\omega\omega^\prime\Cdot\nu} (M \otimes N)_\mathrm{reg} \]
	for all $\nu\in C_\mathrm{fund} \cap X$, and we conclude that
	\begin{align*}
	\big( T_{\omega\Cdot0}^{\omega\Cdot\lambda} M \otimes T_{\omega^\prime\Cdot0}^{\omega^\prime\Cdot\mu} N \big)_\mathrm{reg} & \cong \bigoplus_{\nu\in C_\mathrm{fund} \cap X} \big( T(\nu) \otimes (M \otimes N)_\mathrm{reg} \big)_\mathrm{reg}^{\oplus c_{\lambda,\mu}^\nu} \\
	& \cong \bigoplus_{\nu\in C_\mathrm{fund} \cap X} \big( T_{\omega\omega^\prime\Cdot0}^{\omega\omega^\prime\Cdot\nu} (M \otimes N)_\mathrm{reg} \big)^{\oplus c_{\lambda,\mu}^\nu} ,
	\end{align*}
	as claimed.
\end{proof}

For future applications, let us briefly explain how the action of $\Omega$ can be used to compare the regular parts of tensor products of $\G$-modules with constituents belonging to different linkage classes in the extended principal block $\Rep_{\Omega\Cdot0}(\G)$.
For $\omega\in \Omega$, consider the auto-equivalence
\[ T^\omega \coloneqq \bigoplus_{\lambda\in\Omega\Cdot0} T_\lambda^{\omega\Cdot\lambda} \]
of $\Rep_{\Omega\Cdot0}(\G)$.

\begin{Lemma} \label{lem:regularpartfundamentalgroup}
	Let $M$ and $N$ be $\G$-modules in $\Rep_{\Omega\Cdot0}(\G)$ and let $\omega,\omega^\prime \in \Omega$. Then
	\[ \big( T^\omega M \otimes T^{\omega^\prime} N \big)_\mathrm{reg} \cong T^{\omega\omega^\prime} ( M \otimes N )_{\mathrm{reg}} . \]
\end{Lemma}
\begin{proof}
	We could deduce this as a special case of Theorem \ref{thm:regularparttensorproduct} where $\lambda$ and $\mu$ belong to $\Omega\Cdot0$, but to avoid excessive indexing, we prefer to prove the claim directly. (The reader will note that the proof is also just a special case of the proof of Theorem \ref{thm:regularparttensorproduct}.)
	By Lemma \ref{lem:regularpartdirectsumtensorproduct} and Corollary \ref{cor:regularparttranslation}, we have
	\begin{align*}
		\big( T^\omega M \otimes T^{\omega^\prime} N \big)_\mathrm{reg} & \cong \big( ( T^\omega M )_\mathrm{reg} \otimes ( T^{\omega^\prime} N )_\mathrm{reg} \big)_\mathrm{reg} \\
		& \cong \Big( \big( T(\omega\Cdot0) \otimes M \big)_\mathrm{reg} \otimes \big( T(\omega^\prime\Cdot0) \otimes N \big)_\mathrm{reg} \Big)_\mathrm{reg} \\
		& \cong \big( T(\omega\Cdot0) \otimes M \otimes T(\omega^\prime\Cdot0) \otimes N \big)_\mathrm{reg} \\
		& \cong \Big( \big( T(\omega\Cdot0) \otimes T(\omega^\prime\Cdot0) \big)_\mathrm{reg} \otimes \big( M \otimes N \big)_\mathrm{reg} \Big)_\mathrm{reg} , \quad
	\end{align*}
	where $\big( T(\omega\Cdot0) \otimes T(\omega^\prime\Cdot0) \big)_\mathrm{reg} \cong T(\omega\omega^\prime\Cdot0)$ by Lemma \ref{lem:VerlindeFundamentalGroup}.
	Again using Corollary \ref{cor:regularparttranslation}, we obtain
	 \[ \big( T(\omega\omega^\prime\Cdot0) \otimes ( M \otimes N )_\mathrm{reg} \big)_\mathrm{reg} \cong \big( T^{\omega\omega^\prime} ( M \otimes N )_\mathrm{reg} \big)_\mathrm{reg} \cong T^{\omega\omega^\prime} ( M \otimes N )_\mathrm{reg} , \]
	 and the claim follows.
\end{proof}

\section{Strong regularity} \label{sec:strongregularity}

For certain applications, it may be important to decide if the tensor product $M \otimes N$ of two $\G$-modules $M$ and $N$ is regular.
In the quantum case, it is sufficient to assume that $M$ and $N$ are regular (see Remark \ref{rem:tensorproductregular} below), but we do not know if this is true in the modular case.
To overcome this problem, we introduce the notion of \emph{strong regularity}.
For a non-zero $\G$-module $M$, the \emph{good filtration dimension} of $M$ is defined as
\[ \gfd(M) = \max\big\{ d \mathrel{\big|} \Ext_\G^d(\Delta(\lambda),M) \neq 0 \text{ for some } \lambda \in X^+ \big\} . \]
By Lemma 2.15 in \cite{GruberMinimalTilting}, we also have
\begin{equation} \label{eq:GFDminimaltilitng}
	\gfd(M) = \max\big\{ d \mathrel{\big|} C_\mathrm{min}(M)_d \neq 0 \big\}
\end{equation}

\begin{Definition}
	A non-zero $\G$-module $M$ with $\gfd(M) = d$ is called \emph{strongly regular} if $C_\mathrm{min}(M)_{d-1}$ is negligible and $C_\mathrm{min}(M)_d$ is non-negligible.
\end{Definition}

\begin{Remark} \label{rem:stronglyregularWeylmodulesimplemodule}
	Observe that, for $x \in W_\mathrm{ext}^+$ and $\lambda \in C_\mathrm{fund} \cap X$, the Weyl module $\Delta(x\Cdot\lambda)$ and the simple $\G$-module $L(x\Cdot\lambda)$ are both strongly regular of good filtration dimension $\ell(x)$.
	Indeed, by the description of the minimal tilting complexes of $\Delta(x\Cdot\lambda)$ and $L(x\Cdot\lambda)$ in Propositions \ref{prop:minimalcomplexWeylmodule} and \ref{prop:minimalcomplexsimplemodule}, the tilting modules $C_\mathrm{min}\big( \Delta(x\Cdot\lambda) \big)_i$ and $C_\mathrm{min}\big( L(x\Cdot\lambda) \big)_i$ are negligible for $i = \ell(x)-1$, non-negligible for $i = \ell(x)$ and zero for $i>\ell(x)$.
	The fact that $\Delta(x\Cdot\lambda)$ and $L(x\Cdot\lambda)$ have good filtration dimension $\ell(x)$ has already been observed by A.\ Parker in \cite{ParkerGFD}.
\end{Remark}

Our interest in strongly regular $\G$-modules is founded in the following result:

\begin{Lemma} \label{lem:stronglyregulartensorproduct}
	Let $M$ and $N$ be strongly regular $\G$-modules.
	Then $M \otimes N$ is strongly regular and
	\[ \gfd(M \otimes N) = \gfd(M) + \gfd(N) . \]
\end{Lemma}
\begin{proof}
	Set $d=\gfd(M)$ and $d^\prime=\gfd(N)$, so that $C_\mathrm{min}(M)_i=0$ for $i>d$ and $C_\mathrm{min}(N)_i=0$ for $i>d^\prime$ by equation \eqref{eq:GFDminimaltilitng}.
	Since $M$ and $N$ are strongly regular, there exist $\nu,\nu^\prime\in C_\mathrm{fund} \cap X$ such that
	\[ \big[ C_\mathrm{min}(M)_d: T(\nu) \big]_\oplus \neq 0 \qquad \text{and} \qquad \big[ C_\mathrm{min}(N)_{d^\prime}: T(\nu^\prime) \big]_\oplus \neq 0 , \]
	and by Lemma \ref{lem:Verlindecoefficientnonzero}, there exists $\delta \in C_\mathrm{fund} \cap X$ with $\big[ T(\nu) \otimes T(\nu^\prime) : T(\delta) \big]_\oplus \neq 0$.
	Thus $T(\delta)$ appears as a direct summand of the tensor product $C_\mathrm{min}(M)_d \otimes C_\mathrm{min}(N)_{d^\prime}$, which is the degree $d+d^\prime$ term of the tensor product complex $C_\mathrm{min}(M) \otimes C_\mathrm{min}(N)$.
	Furthermore, the degree $d+d^\prime-1$ term
	\[ \big( C_\mathrm{min}(M)_{d-1} \otimes C_\mathrm{min}(N)_{d^\prime} \big) \oplus \big( C_\mathrm{min}(M)_d \otimes C_\mathrm{min}(N)_{d^\prime-1} \big) \]
	of the tensor product complex is negligible, and the terms in degree $i>d+d^\prime$ of the tensor product complex are zero.
	Now $C_\mathrm{min}(M \otimes N)$ is the minimal complex of $C_\mathrm{min}(M) \otimes C_\mathrm{min}(N)$ by part(4) of Lemma \ref{lem:minimaltiltingcomplex}, hence $C_\mathrm{min}(M \otimes N)_{d+d^\prime-1}$ is negligible and $C_\mathrm{min}(M \otimes N)_i = 0$ for $i>d+d^\prime$.
	As the terms of the tensor product complex $C_\mathrm{min}(M) \otimes C_\mathrm{min}(N)$ in degrees $d+d^\prime-1$ and $d+d^\prime+1$ are negligible or zero, respectively, 
	Corollary 2.8 in \cite{GruberMinimalTilting}
	implies that
	\[ 0 \neq \big[ C_\mathrm{min}(M)_d \otimes C_\mathrm{min}(N)_{d^\prime} : T(\delta) \big]_\oplus = \big[ C_\mathrm{min}(M \otimes N)_{d+d^\prime} : T(\delta) \big]_\oplus . \]
	Finally, equation \eqref{eq:GFDminimaltilitng} yields $\gfd(M \otimes N) = d + d^\prime$, and it follows that $M \otimes N$ is strongly regular.
\end{proof}

\begin{Remark} \label{rem:tensorproductregular}
	In the quantum case, we claim that the tensor product $M \otimes N$ of two regular $\G$-modules $M$ and $N$ is always regular, if $\ell>h$.
	Observe that the claim is equivalent to the statement that singular $\G$-modules form a \emph{prime ideal}, i.e.\ that the tensor product $M \otimes N$ of two $\G$-modules $M$ and $N$ is singular only if at least one of $M$ and $N$ is singular.
	By \cite[Lemma 2.8]{GruberTensorIdeals} and its proof, we have
	\[ \langle \mathcal{N} \rangle = \bigcap_{ \langle \mathcal{N} \rangle \subseteq \mathcal{P} } \mathcal{P} = \bigcap_{ \langle \mathcal{N} \rangle \subseteq \mathcal{P} } \langle \mathcal{P} \cap \Tilt(\G) \rangle , \]
	where the intersection runs over the prime thick tensor ideals $\mathcal{P}$ in $\Rep(\G)$ with the 2/3-property such that $\langle \mathcal{N} \rangle \subseteq \mathcal{P}$.
	For any such tensor ideal $\mathcal{P}$, we have $\mathcal{N} \subseteq \mathcal{P} \cap \Tilt(\G)$, and as $\mathcal{N}$ is maximal among the proper thick tensor ideals in $\Tilt(\G)$ (by Lemma \ref{lem:Verlindecoefficientnonzero}), it follows that $\mathcal{N} = \mathcal{P} \cap \Tilt(\G)$ and
	\[ \langle \mathcal{N} \rangle = \langle \mathcal{P} \cap \Tilt(\G) \rangle = \mathcal{P} , \]
	again by Lemma 2.8 in \cite{GruberTensorIdeals}.
	Hence $\langle \mathcal{N} \rangle$ is prime, as required.
	We do not know if the analogous statement is true in the modular case.
\end{Remark}

Lemma \ref{lem:stronglyregulartensorproduct} allows us to prove the following generalization to tensor products of A.~Parker's results from \cite{ParkerGFD} about the good filtration dimension of Weyl modules and simple $\G$-modules:

\begin{Theorem} \label{thm:tensorproductweylmodulesimplemoduleGFD}
	Let $x_1,\ldots,x_m,y_1,\ldots,y_n\in W_\mathrm{ext}^+$ and $\lambda_1,\ldots,\lambda_m,\mu_1,\ldots,\mu_n\in C_\mathrm{fund} \cap X$.
	Then the tensor product
	\[ \Delta(x_1\Cdot\lambda_1) \otimes \cdots \otimes \Delta(x_m\Cdot\lambda_m) \otimes L(y_1\Cdot\mu_1) \otimes \cdots \otimes L(y_n\Cdot\mu_n) \]
	is strongly regular and has good filtration dimension $\ell(x_1) + \cdots + \ell(x_m) + \ell(y_1) + \cdots + \ell(y_n)$.
\end{Theorem}
\begin{proof}
	For $1 \leq i \leq m$ and $1 \leq j \leq n$, the $\G$-modules $\Delta(x_i\Cdot\lambda_i)$ and $L(y_j\Cdot\mu_j)$ are strongly regular by Remark \ref{rem:stronglyregularWeylmodulesimplemodule}, and their good filtration dimensions are given by
	\[ \gfd\big( \Delta(x_i\Cdot\lambda_i) \big) = \ell(x_i) \qquad \text{and} \qquad \gfd\big( L(y_j\Cdot\lambda_j) \big) = \ell(y_j) . \]
	The claim follows from Lemma \ref{lem:stronglyregulartensorproduct}, by induction on $m+n$.
\end{proof}

\section{An application to translation functors} \label{sec:translation}

In this section, we give a further application of the tensor ideal $\langle \mathcal{N} \rangle$ of singular $\G$-modules, which may be of independent interest.
We prove that the composition of two translation functors between $\ell$-regular linkage classes is naturally isomorphic to a translation functor.
This statement should not be very surprising to experts in the field, but we are not aware of a proof in the literature.

\begin{Proposition} \label{prop:compositionoftranslationfunctors}
	Let $\lambda,\mu,\delta\in C_\mathrm{fund} \cap X$. Then there is an isomorphism of functors $T_\delta^\mu \cong T_\lambda^\mu \circ T_\delta^\lambda$.
\end{Proposition}
\begin{proof}
	First suppose that $\delta=0$ and recall that
	\[ T_0^\lambda = \pr_\lambda \big( \nabla(\lambda) \otimes - \big) , \qquad T_0^\mu = \pr_\mu \big( \nabla(\mu) \otimes - \big) \qquad \text{and} \qquad T_\lambda^\mu = \pr_\mu \big(  \nabla(\nu) \otimes - \big) , \]
	where $\nu$ is the unique dominant weight in the $W_\mathrm{fin}$-orbit of $\mu-\lambda$.
	Consider the functor
	\[ \Psi \coloneqq \pr_\mu \big( \nabla(\nu) \otimes \nabla(\lambda) \otimes - \big) , \]
	and note that the canonical embedding of functors $T_0^\lambda \Longrightarrow \big( \nabla(\lambda) \otimes - \big)$ gives rise to a natural transformation $T_\lambda^\mu \circ T_0^\lambda \Longrightarrow \Psi$.
	Furthermore, we have
	\[ \nabla(\mu) \cong T_\lambda^\mu \nabla(\lambda) = \pr_\mu \big( \nabla(\nu) \otimes \nabla(\lambda) \big) , \]
	and the canonical projection $\nabla(\nu) \otimes \nabla(\lambda) \to \nabla(\mu)$ affords a natural transformation $\Psi \Longrightarrow T_0^\mu$.
	We claim that the composition of these natural transformations
	\[ \vartheta \colon T_\lambda^\mu \circ T_0^\lambda \Longrightarrow \Psi \Longrightarrow T_0^\mu \]
	is a natural isomorphism.
	
	 Let $N$ be a complement of $\nabla(\mu) \cong \pr_\mu \big( \nabla(\nu) \otimes \nabla(\lambda) \big)$ in $\nabla(\lambda) \otimes \nabla(\nu)$, and observe that $\pr_\mu N = 0$.
	 For a $\G$-module $M$ in $\Rep_0(\G)$, we have
	\[ ( M \otimes N )_\mathrm{reg} \cong \bigoplus_{\nu \in C_\mathrm{fund} \cap X} (M \otimes \pr_\nu N)_\mathrm{reg} \]
	by the linkage principle and Lemma \ref{lem:singularidealsingularlinkageclasses}, and as $(M \otimes \pr_\nu N)_\mathrm{reg}$ belongs to $\Rep_\nu(\G)$ by Lemma \ref{lem:regularpartlinkageclass}, we conclude that
	\[ \pr_\nu( M \otimes N )_\mathrm{reg} \cong ( M \otimes \pr_\nu N )_\mathrm{reg} \]
	for all $\nu \in C_\mathrm{fund} \cap X$.
	 In particular, the functor $\pr_\mu( N \otimes - )$ maps every $\G$-module in $\Rep_0(\G)$ into the tensor ideal $\langle \mathcal{N} \rangle$ of singular $\G$-modules.
	 As $\Psi$ decomposes as the direct sum of the functors $T_0^\mu$ and $\pr_\mu( N \otimes - )$, this implies that all components of the natural transformation $\Psi \Longrightarrow T_0^\mu$ descend to isomorphisms in the regular quotient $\regRep(\G)$.
	 Similarly, the embedding of functors
	 \[ T_0^\lambda \Longrightarrow \big( \nabla(\lambda) \otimes - \big) \cong \big( T(\lambda) \otimes - \big)\]
	 descends to a natural isomorphism in $\regRep(\G)$ by Corollary \ref{cor:translationquotient}, and it follows that the same is true for the natural transformation
	 \[ T_\lambda^\mu \circ T_0^\lambda \Longrightarrow \pr_\mu\big( \nabla(\nu) \otimes - \big) \circ \big( \nabla(\lambda) \otimes - \big) = \Psi . \]
	 In particular, the component of $\vartheta$ at any simple $\G$-module $L(x\Cdot0)$ with $x\in W_\mathrm{aff}^+$ affords an isomorphism in $\regRep(\G)$.
	 Now $L(x\Cdot0)$ is non-zero in $\regRep(\G)$ by Lemma \ref{lem:regularquotientregularweight}, whence the endomorphism algebra of $L(x\Cdot0)$ in $\regRep(\G)$ is also non-zero.
	 Since the latter endomorphism algebra is a quotient of the endomorphism algebra of $L(x\Cdot0)$ in $\Rep(\G)$,
	 we conclude that the component of $\vartheta$ at $L(x\Cdot0)$ is non-zero;
	 hence it affords an isomorphism between $T_\lambda^\mu T_0^\lambda L(x\Cdot0) \cong L(x\Cdot\mu)$ and $T_0^\mu L(x\Cdot0) \cong L(x\Cdot\mu)$, by Schur's Lemma.
	 Using the snake Lemma and induction on the length of a composition series, one easily deduces that the component of $\vartheta$ at every $\G$-module in $\Rep_0(\G)$ is an isomorphism, so $\vartheta$ is a natural isomorphism, as claimed.
	 
	 Now since $T_0^\lambda \circ T_\lambda^0$ is isomorphic to the identity functor on $\Rep_\lambda(\G)$, we further obtain isomorphisms of functors
	\[ T_\lambda^\mu \cong T_\lambda^\mu \circ T_0^\lambda \circ T_\lambda^0 \cong T_0^\mu \circ T_\lambda^0 . \]
	For arbitrary $\delta\in C_\mathrm{fund} \cap X$, we conclude that
	\[ T_\mu^\delta \circ T_\lambda^\mu \cong T_0^\delta \circ T_\mu^0 \circ T_0^\mu \circ T_\lambda^0 \cong T_0^\delta \circ T_\lambda^0 \cong T_\lambda^\delta , \]
	as required.
\end{proof}

\section{Tensor product decomposition of \texorpdfstring{$\regRep(\G)$}{the regular quotient}}
\label{sec:tensorproductdecomposition}

In this section, we explain how the results from Section \ref{sec:singularmodules} give rise to an external tensor product decomposition
\[ \regRep(\G) \cong \Ver(\G) \boxtimes \regRep_0(\G) \]
of the regular quotient.
We first explain what we mean by the external tensor product of two additive $\kk$-linear categories.

Let $\mathcal{A}$ and $\mathcal{B}$ be two additive $\kk$-linear categories.
Following \cite[Section 1.4]{KellyEnriched}, we define the tensor product $\mathcal{A} \otimes \mathcal{B}$ as the $\kk$-linear category with objects $\mathrm{Ob}(\mathcal{A} \otimes \mathcal{B}) = \mathrm{Ob}(\mathcal{A}) \times \mathrm{Ob}(\mathcal{B})$ and homomorphisms
\[ \Hom_{\mathcal{A} \otimes \mathcal{B}}( A \boxtimes B , A^\prime \boxtimes B^\prime ) = \Hom_\mathcal{A}(A,A^\prime) \otimes \Hom_\mathcal{B}(B,B^\prime) \]
for $A , A^\prime \in \mathrm{Ob}(\mathcal{A})$ and $B , B^\prime \in \mathrm{Ob}(\mathcal{B})$,
where we write $A \boxtimes B$ for the object of $\mathcal{A} \otimes \mathcal{B}$ corresponding to $A$ and $B$.
The composition of homomorphisms in $\mathcal{A} \otimes \mathcal{B}$ is induced by the composition in $\mathcal{A}$ and $\mathcal{B}$ in the obvious way.
The category $\mathcal{A} \otimes \mathcal{B}$ need not be additive, so we define $\mathcal{A} \boxtimes \mathcal{B}$ to be its additive envelope.%
\footnote{This means that objects of $\mathcal{A} \boxtimes \mathcal{B}$ are finite tuples of objects of $\mathcal{A} \otimes \mathcal{B}$, and homomorphisms are given by matrices whose entries are homomorphisms in $\mathcal{A} \otimes \mathcal{B}$.}
Thus $\mathcal{A} \boxtimes \mathcal{B}$ is an additive $\kk$-linear category.
If $\mathcal{A}$ and $\mathcal{B}$ are endowed with $\kk$-linear monoidal structures then there is a canonical $\kk$-linear monoidal structure on $\mathcal{A} \otimes \mathcal{B}$, and the latter extends uniquely to a $\kk$-linear monoidal structure on $\mathcal{A} \boxtimes \mathcal{B}$.

Now let $\mathcal{C}$ be another $\kk$-linear category.
According to Section 1.5 in \cite{KellyEnriched}, a $\kk$-linear functor 
\[ F \colon \mathcal{A} \otimes \mathcal{B} \longrightarrow \mathcal{C} \]
can equivalently be defined by two families of $\kk$-linear functors $F_A \colon \mathcal{B} \to \mathcal{C}$ and $G_B \colon \mathcal{A} \to \mathcal{C}$, indexed by $\mathrm{Ob}(\mathcal{A})$ and $\mathrm{Ob}(\mathcal{B})$, respectively, such that
\[ F_A(B) = G_B(A) = F(A \boxtimes B) \]
and such that the following diagram commutes, for all $A,A^\prime \in \mathrm{Ob}(\mathcal{A})$ and $B,B^\prime \in \mathrm{Ob}(\mathcal{B})$.
\begin{equation} \label{eq:functorfromtensorproduct}
\begin{tikzpicture}[scale=.85,every node/.style={scale=0.85},baseline=(current  bounding  box.center)]
	\node (A1) at (0,0) {$\Hom_\mathcal{A}(A,A^\prime) \otimes \Hom_\mathcal{B}(B,B^\prime)$};
	\node (A2) at (0,-4) {$\Hom_\mathcal{B}(B,B^\prime) \otimes \Hom_\mathcal{A}(A,A^\prime)$};
	\node (A3) at (10,0) {$\Hom_\mathcal{C}\big( F(A \boxtimes B^\prime) , F(A^\prime \boxtimes B^\prime) \big) \otimes \Hom_\mathcal{C}\big( F(A \boxtimes B) , F(A \boxtimes B^\prime) \big)$};
	\node (A4) at (10,-2) {$\Hom_\mathcal{C}\big( F(A \boxtimes B) , F(A^\prime \boxtimes B^\prime) \big)$};
	\node (A5) at (10,-4) {$\Hom_\mathcal{C}\big( F(A^\prime \boxtimes B) , F(A^\prime \boxtimes B^\prime) \big) \otimes \Hom_\mathcal{C}\big( F(A \boxtimes B) , F(A^\prime \boxtimes B) \big)$};
	
	\draw[->] (A1) -- (A2);
	\draw[->] (A1) -- node[above] {\footnotesize $G_{B^\prime} \otimes F_A$} (A3);
	\draw[->] (A2) -- node[below] {\footnotesize $F_{A^\prime} \otimes G_B$} (A5);
	\draw[->] (A3) -- (A4);
	\draw[->] (A5) -- (A4);
\end{tikzpicture}
\end{equation}
The vertical arrow on the left hand side is induced by the braiding in the category of $\kk$-vector spaces and the vertical arrows on the right hand side denote the composition of morphisms in $\mathcal{C}$.

By the above discussion, the additive $\kk$-linear category $\Ver(\G) \boxtimes \regRep_0(\G)$ has a $\kk$-linear monoidal structure.
For any tilting $\G$-module $T$, the functor $(T \otimes -) \colon \Rep_0(\G) \to \Rep(\G)$ induces a functor
\[ F_T = ( T \otimes - ) \colon \regRep_0(\G) \longrightarrow \regRep(\G) \]
because singular $\G$-modules form a thick tensor ideal, and similarly, for any $\G$-module $M$ in $\Rep_0(\G)$, the functor $(- \otimes M) \colon \Tilt(\G) \to \Rep(\G)$ induces a functor
\[ G_M = ( - \otimes M ) \colon \Ver(\G) \longrightarrow \regRep(\G) \]
because negligible tilting modules are singular.
We have
\[ F_T(M) = G_M(T) = T \otimes M , \]
and the diagram \eqref{eq:functorfromtensorproduct} commutes because
\[ ( \id_{T^\prime} \otimes g ) \circ ( f \otimes \id_M ) = f \otimes g = ( f \otimes \id_{M^\prime} ) \circ ( \id_T \otimes g ) \]
for all homomorphisms $f \colon T \to T^\prime$ in $\Tilt(\G)$ and $g \colon M \to M^\prime$ in $\Rep_0(G)$.
Hence there exists a unique $\kk$-linear functor
\[ F \colon \Ver(\G) \otimes \regRep_0(\G) \longrightarrow \regRep(\G) \]
such that
\[ F( T \boxtimes - ) = ( T \otimes - ) \qquad \text{and} \qquad F( - \boxtimes M ) = ( - \otimes M ) \]
for all tilting $\G$-modules $T$ and all $\G$-modules $M$ in $\Rep_0(\G)$.
Furthermore, the braiding on $\Rep(\G)$ defines a natural isomorphism $F( - \otimes - ) \cong F( - ) \otimes F( - )$ with components
\begin{multline*}
F\big( ( T \boxtimes M ) \otimes ( T^\prime \boxtimes M^\prime ) \big) = F\big( ( T \boxtimes T^\prime ) \otimes ( M \boxtimes M^\prime ) \big) = T \otimes T^\prime \otimes M \otimes M^\prime \\ \cong T \otimes M \otimes T^\prime \otimes M^\prime = F( T \boxtimes M ) \otimes F( T^\prime \boxtimes M^\prime )
\end{multline*}
for tilting $\G$-modules $T$ and $T^\prime$ and $\G$-modules $M$ and $M^\prime$ in $\Rep_0(\G)$, and we have
\[ F( \kk \boxtimes \kk ) = \kk \otimes \kk \cong \kk . \]
It is straightforward to check that these coherence maps endow $F$ with the structure of a (strong) monoidal functor.

Since $\regRep(\G)$ is an additive $\kk$-linear category, the functor $F$ extends uniquely to an additive $\kk$-linear functor
\[ \hat F \colon \Ver(\G) \boxtimes \regRep_0(\G) \longrightarrow \regRep(\G) , \]
and the coherence maps for $F$ extend canonically to coherence maps for $\hat F$, so that $\hat F$ becomes a monoidal functor.

\begin{Theorem} \label{thm:regularquotient}
	The functor
	\[ \hat F \colon \Ver(\G) \boxtimes \regRep_0(\G) \longrightarrow \regRep(\G) , \]
	together with the canonical coherence maps, is an equivalence of $\kk$-linear monoidal categories.
\end{Theorem}
\begin{proof}
	Recall that for $\lambda \in C_\mathrm{fund} \cap X$, we write $\regRep_\lambda(\G)$ for the essential image of $\Rep_\lambda(\G)$ under the quotient functor $q \colon \Rep(\G) \to \regRep(\G)$.
	The projection functor $\pr_\lambda \colon \Rep(\G) \to \Rep_\lambda(\G)$ induces a functor $\pr_\lambda \colon \regRep(\G) \to \regRep_\lambda(\G)$ and the translation functor $T_0^\lambda$ induces an equivalence between $\regRep_0(\G)$ and $\regRep_\lambda(\G)$ with quasi inverse induced by $T_\lambda^0$.
	We consider the functor
	\[ G = \bigoplus_{\lambda \in C_\mathrm{fund} \cap X} T(\lambda) \boxtimes T_\lambda^0 \circ \pr_\lambda( - ) \; \colon \quad \regRep(\G) \longrightarrow \Ver(\G) \boxtimes \regRep_0(\G) \]
	and claim that $G$ is a quasi-inverse for $\hat F$.
	Indeed, we have an isomorphism of functors
	\begin{multline*}
	\quad \hat F \circ G
	= \hat F \circ \Big( \bigoplus_\lambda T(\lambda) \boxtimes T_\lambda^0 \circ \pr_\lambda( - ) \Big)
	= \bigoplus_\lambda T(\lambda) \otimes T_\lambda^0 \circ \pr_\lambda( - ) \\
	\cong \bigoplus_\lambda T_0^\lambda \circ T_\lambda^0 \circ \pr_\lambda( - )\
	\cong \bigoplus_\lambda \pr_\lambda( - )
	\cong \id_{\regRep(\G)} , \quad
	\end{multline*}
	where the first two equality signs follow from the definitions.
	The first isomorphism follows from Corollary \ref{cor:translationquotient} because $\big( T(\lambda) \otimes - \big) \cong T_0^\lambda$ as functors from $\regRep_0(\G)$ to $\regRep(\G)$, the second isomorphism follows from the fact that $T_0^\lambda$ is a quasi-inverse of $T_\lambda^0$, and the third isomorphism follows from the linkage principle and the fact that all $\G$-modules in $\ell$-singular linkage classes are singular (see Lemma \ref{lem:singularidealsingularlinkageclasses}).
	Conversely, for $\lambda \in C_\mathrm{fund} \cap X$ and $M$ a $\G$-module in $\Rep_0(\G)$, we have
	\[ G \circ \hat F\big( T(\lambda) \boxtimes M \big) = G\big( T(\lambda) \otimes M \big) \cong G( T_0^\lambda M ) = T(\lambda) \boxtimes T_\lambda^0 T_0^\lambda M \cong T(\lambda) \boxtimes M \]
	in $\Ver(\G) \boxtimes \regRep_0(\G)$, again by Corollary \ref{cor:translationquotient}.
	As $\Ver(\G)$ is semisimple with simple objects $T(\lambda)$, for $\lambda \in C_\mathrm{fund} \cap X$, this gives rise to a natural isomorphism between $G \circ \hat F$ and the identity functor on $\Ver(\G) \boxtimes \regRep_0(\G)$, whence $G$ is a quasi inverse of $\hat F$, as claimed.
	Thus $\hat F$ is an equivalence of categories, and even an equivalence of monoidal categories since $\hat F$ is monoidal (see Proposition 4.4.2 in \cite{Rivano}).
\end{proof}

\begin{Remark}
	Let us write $\Ver_{\Z\Phi}(\G)$ for the full subcategory of $\Ver(\G)$ whose objects are the direct sums of indecomposable tilting modules with highest weights in $\Z\Phi$, and $\Ver_\Omega(\G)$ for the full subcategory of $\Ver(\G)$ whose objects are isomorphic to direct sums of tilting modules of the form $T(\omega\Cdot0)$ for certain $\omega \in \Omega$.
	Then $\Ver_{\Z\Phi}(\G)$ is closed under tensor products in $\Ver(\G)$ because $\Z\Phi$ is a subgroup of $X$ and $\Ver_\Omega(\G)$ is closed under tensor products in $\Ver(\G)$ by Lemma \ref{lem:VerlindeFundamentalGroup}.
	There is an equivalence of $\kk$-linear monoidal categories
	\begin{equation} \label{eq:Verlindedecomposition}
	\Ver(\G) \cong \Ver_{\Z\Phi}(\G) \boxtimes \Ver_\Omega(\G) .
	\end{equation}
	Arguing as in the proof of Theorem \ref{thm:regularquotient}, we further obtain equivalences of $\kk$-linear monoidal categories
	\[ \regRep(\G) \cong \Ver_{\Z\Phi}(\G) \boxtimes \regRep_{\Omega\Cdot0}(\G) \qquad \text{and} \qquad \regRep_{\Omega\Cdot0}(\G) \cong \Ver_\Omega(\G) \boxtimes \regRep_0(\G) . \]
	Note that the latter gives a categorical interpretation of Lemma \ref{lem:regularpartfundamentalgroup}.
	The above equivalences combine to an equivalence of $\kk$-linear monoidal categories
	\[ \regRep(\G) \cong \Ver_{\Z\Phi}(\G) \boxtimes \Ver_\Omega(\G) \boxtimes \regRep_0(\G) , \]
	which can also be obtained from Theorem \ref{thm:regularquotient} and equation \eqref{eq:Verlindedecomposition}
\end{Remark}

\section{Examples} \label{sec:examples}

In this Section, let $\G$ be of type $\mathrm{A}_2$, with simple roots $\Pi = \{ \alpha_1 , \alpha_2 \}$ .
Let $\varpi_1 , \varpi_2 \in X^+$ be the fundamental dominant weights with $(\varpi_i,\alpha_j^\vee) = \delta_{ij}$, and let $s_1 = s_{\alpha_1}$ and $s_2 = s_{\alpha_2}$ be the simple reflections in $W_\mathrm{fin}$.
The affine Weyl group $W_\mathrm{aff}$ is generated by the simple reflections $S = \{ s_0 , s_1 , s_2 \}$, where $s_0 = t_{\alpha_\mathrm{h}} s_{\alpha_\mathrm{h}}$ and $\alpha_\mathrm{h} = \alpha_1 + \alpha_2$.
For all $x \in W_\mathrm{aff}^+$, we have
\[ L(x \Cdot 0) \otimes L(0) \cong L(x \Cdot 0) \cong \big( L(x \Cdot 0) \otimes L(0) \big)_\mathrm{reg}  \]
because $L(x\Cdot0)$ is regular by Lemma \ref{lem:regularquotientregularweight}, and using Theorem \ref{thm:regularparttensorproduct}, it follows that
\[ \big( L(x\Cdot\lambda) \otimes L(\mu) \big)_\mathrm{reg} \cong \bigoplus_{\nu \in C_\mathrm{fund} \cap X} T_0^\nu L(x\Cdot0)^{\oplus c_{\lambda,\mu}^\nu} \cong \bigoplus_{\nu \in C_\mathrm{fund} \cap X} L(x\Cdot\nu)^{\oplus c_{\lambda,\mu}^\nu} \]
for $\lambda,\mu \in C_\mathrm{fund} \cap X$.
So far, we have not used the fact that $\G$ is of type $\mathrm{A}_2$.
Let us next discuss a less trivial example.

A weight $\lambda \in X^+$ is called $\ell$-restricted if $\lambda = a \varpi_1 + b \varpi_2$ with $0 \leq a,b \leq p-1$.
The set of indecomposable direct summands of tensor products of simple $\G$-modules with $\ell$-restricted highest weight has been determined by C.\ Bowman, S.\ Doty and S.\ Martin in \cite{BowmanDotyMartin}.
They show that every indecomposable direct summand of such a tensor product is either an indecomposable tilting module or a $\G$-module of the form $M(\nu)$, for $\nu \in C_\mathrm{fund} \cap X$, whose radical series (and socle series) is displayed in the Alperin diagram below.%
\footnote{In \cite{BowmanDotyMartin}, the $\G$-module $M(\nu)$ is labeled by $s_0\Cdot\nu$  instead of $\nu \in C_\mathrm{fund} \cap X$.}
(We replace simple modules $L(x\Cdot\nu)$ by their label $x \in W_\mathrm{aff}^+$.)
\[ M(\nu) = 
	\begin{tikzpicture}[scale=.6,baseline={([yshift=-.5ex]current bounding box.center)}]
	\node (C1) at (0,0) {\small $s_0$};
	\node (C2) at (-1.5,1.5) {\small $s_0s_1$};
	\node (C3) at (0,1.5) {\small $e$};
	\node (C4) at (1.5,1.5) {\small $s_0s_2$};
	\node (C5) at (0,3) {\small $s_0$};
	
	\draw (C1) -- (C2);
	\draw (C1) -- (C3);
	\draw (C1) -- (C4);
	\draw (C2) -- (C5);
	\draw (C3) -- (C5);
	\draw (C4) -- (C5);
	\end{tikzpicture} \]
A straightforward character computation shows that in the Grothendieck group of $\Rep(\G)$, we have
\[ [ M(\nu) ] = [ T(s_0s_1\Cdot\nu) ] + [ T(s_0s_2\Cdot\nu) ] - 2 \cdot [ T(s_0\Cdot\nu) ] + 3 \cdot [ T(\nu) ] . \]
Since we also have $[M(\nu)] = \sum_i (-1)^i \cdot \big[ C_\mathrm{min}\big(M(\nu)\big)_i \big]$, it follows that $M(\nu)$ is regular.

With $\omega = t_{\varpi_1} s_1 s_2$, we have $\Omega = \{ e , \omega , \omega^{-1} \}$ and
\[ \omega \Cdot 0 = (\ell-3) \cdot \varpi_1 , \qquad \omega^{-1} \Cdot 0 = (\ell-3) \cdot \varpi_2 . \]
According to Sections 8.1--8.6 in \cite{BowmanDotyMartin}, we have
\[ \pr_{\omega^{-1}\Cdot0}\big( L(s_0\omega\Cdot0) \otimes L(s_0\omega\Cdot0) \big) \cong M( \omega^{-1} \Cdot 0 ) \oplus L( \omega^{-1} \Cdot 0 ) , \]
and as both $M( \omega^{-1} \Cdot 0 )$ and $L( \omega^{-1} \Cdot 0 )$ are regular, Lemmas \ref{lem:regularpartlinkageclass} and \ref{lem:regularpartfundamentalgroup} imply that
\begin{equation} \label{eq:regularpartA2s0s0}
\big( L(s_0\Cdot0) \otimes L(s_0\Cdot0) \big)_\mathrm{reg} \cong M(0) \oplus L(0) .
\end{equation}
Now Theorem \ref{thm:regularparttensorproduct} yields
\[ \big( L(s_0 \Cdot \lambda) \otimes L(s_0 \Cdot \mu) \big)_\mathrm{reg} \cong \bigoplus_{\nu \in C_\mathrm{fund} \cap X} \big( M(\nu) \oplus L(\nu) \big)^{\oplus c_{\lambda,\mu}^\nu} \]
for all $\lambda,\mu \in C_\mathrm{fund} \cap X$.

Next observe that we have $\ell \varpi_1 = t_{\varpi_1} \Cdot 0 = s_0 s_2 \omega \Cdot 0$ and $\ell \varpi_2 = t_{\varpi_2} \Cdot 0 = s_0 s_1 \omega \Cdot 0$.
If $\ell \geq 5$ then  it is straightforward to see using the Frobenius twist functor (see Sections II.3.16 and II.H.7 in \cite{Jantzen}) that
\[ L(\ell\varpi_1) \otimes L(\ell\varpi_2) \cong L(\ell\alpha_\mathrm{h}) \oplus L(0) , \]
where $\ell\alpha_\mathrm{h} = t_{\alpha_\mathrm{h}} \Cdot 0 = s_0 s_1 s_2 s_1 \Cdot 0$,
and arguing as before, we see that
\[ \big( L(s_0s_1\Cdot\lambda) \otimes L(s_0s_2\Cdot\mu) \big)_\mathrm{reg} \cong \bigoplus_{\nu \in C_\mathrm{fund} \cap X} \big( L(\nu+\ell\alpha_\mathrm{h}) \oplus L(\nu) \big)^{\oplus c_{\lambda,\mu}^\nu} \]
for $\lambda,\mu \in C_\mathrm{fund} \cap X$.
Analogously, the tensor product decomposition
\[ L(\ell\varpi_1) \otimes L(\ell\varpi_1) \cong L(2\ell\varpi_1) \oplus L(\ell\varpi_2) \]
can be used to show that
\[ \big( L(s_0s_1\Cdot\lambda) \otimes L(s_0s_1\Cdot\mu) \big)_\mathrm{reg} \cong \bigoplus_{\nu \in C_\mathrm{fund} \cap X} \big( L(s_0s_1s_2s_0 \Cdot \nu) \oplus L(s_0s_2 \Cdot \nu) \big)^{\oplus c_{\lambda,\mu}^\nu} \]
for $\lambda,\mu \in C_\mathrm{fund} \cap X$, even if we drop the additional assumption $\ell \geq 5$.

\begin{Remark}
	We call the $\G$-module $M(0)$ in \eqref{eq:regularpartA2s0s0} a \emph{generic direct summand} of $L(s_0\Cdot0) \otimes L(s_0\Cdot0)$, because $M(\nu) \cong T_0^\nu M(0)$ appears \emph{generically} in Krull--Schmidt decompositions of $L(s_0\Cdot\lambda) \otimes L(s_0\Cdot\mu)$, for $\lambda,\mu \in C_\mathrm{fund} \cap X$, with multiplicity $c_{\lambda,\mu}^\nu$.
	More details about generic direct summands and further examples will appear in a forthcoming paper.
\end{Remark}

\bibliographystyle{alpha}
\bibliography{tensor}

\end{document}